\documentclass[12pt,reqno]{amsart}

\usepackage{amsmath}
\usepackage{amssymb}
\usepackage{amsthm}
\usepackage{amscd}
\usepackage{xcolor, graphicx}
\usepackage{xypic}
\usepackage{tikz}
\usetikzlibrary{matrix,arrows}
\usepackage{delarray}
\usepackage{hyperref}
\headsep=1cm \numberwithin{equation}{section}
\def\N{{\mathbb N}}

\def\P{{\mathbb P}}

\newtheorem{theorem}{Theorem}[section]
\newtheorem{lemma}[theorem]{Lemma}

\newtheorem{proposition}[theorem]{Proposition}
\newtheorem{corollary}[theorem]{Corollary}

\theoremstyle{definition}

\newtheorem{remark}[theorem]{Remark}

\newtheorem{convention and reminder}[theorem]{Convention and Reminder}
\newtheorem{convention and remark}[theorem]{Convention and Remark}
\newtheorem{definition and remark}[theorem]{Definition and Remark}

\newtheorem{reminders and definition}[theorem]{Reminders and Definition}

\newtheorem{notation and remarks}[theorem]{Notation and Remarks}
\newtheorem{notation and remark}[theorem]{Notation and Remark}
\newtheorem{example}[theorem]{Example}
\newtheorem{conjecture}[theorem]{Conjecture}

\newcommand\Ker{\operatorname{\Ker}}

\newcommand\Supp{\operatorname{Supp}}

\newcommand\PGL{\operatorname{PGL}}

\def\kk{{\Bbbk}}
\def\QR{\mathsf{QR}}
\def\1{\vec{1}}
\def\0{\vec{0}}

\title[Rank 3 quadratic generators of Veronese embeddings]{Rank 3 quadratic generators of Veronese embeddings}

\author[K. Han]{Kangjin Han}
\email{kjhan@dgist.ac.kr}
\address{
School of Undergraduate Studies, Daegu-Gyeongbuk Institute of
Science \& Technology (DGIST), Daegu 42988, Republic of Korea}

\author[W. Lee]{Wanseok Lee}
\email{wslee@pknu.ac.kr}
\address{Department of Applied Mathematics, Pukyong National University, Busan 608-737, Republic of Korea}

\author[H. Moon]{Hyunsuk Moon}
\email{mhs@kaist.ac.kr}
\address{Korea Advanced Institute of Science and Technology, Daejeon 34141, Republic of Korea}

\author[E. Park]{Euisung Park}
\email{euisungpark@korea.ac.kr}
\address{Department of Mathematics, Korea University, Seoul 136-701, Republic of Korea}

\begin{document}


\begin{abstract}
Let $L$ be a very ample line bundle on a projective
scheme $X$ defined over an algebraically closed field $\Bbbk$ with
${\rm char}~\Bbbk \neq 2$. We say that $(X,L)$ satisfies
\textit{property} $\QR(k)$ if the homogeneous ideal of the linearly
normal embedding $X \subset \P H^0 (X,L)$ can be generated by quadrics
of rank $\leq k$. Many classical varieties such as Segre-Veronese embeddings, rational normal scrolls and curves of
high degree satisfy property $\QR(4)$.

In this paper, we first prove that if ${\rm char}~\Bbbk \neq 3$ then $(\P^n , \mathcal{O}_{\P^n} (d))$ satisfies property $\QR(3)$ for all $n \geq 1$ and $d \geq 2$.
We also investigate an asymptotic behavior of property $\QR(3)$ for any projective scheme. Namely,
we prove that $(i)$ if $X \subset \P H^0 (X,L)$ is $m$-regular then $(X,L^d )$ satisfies property $\QR(3)$
for all $d \geq m$ and $(ii)$ if $A$ is an ample line bundle on $X$ then $(X,A^d )$ satisfies property $\QR(3)$
for all sufficiently large even number $d$. These results provide an affirmative evidence for the expectation that
property $\QR(3)$ holds for all sufficiently ample line bundles on $X$, as in the cases of Green-Lazarsfeld's
condition $\mathrm{N}_p$ and Eisenbud-Koh-Stillman's determininantal presentation in \cite{EKS}.
Finally, when ${\rm char}~\Bbbk = 3$ we prove that $(\P^n , \mathcal{O}_{\P^n} (2))$ fails to satisfy
property $\QR(3)$ for all $n \geq 3$.
\end{abstract}

\maketitle \tableofcontents \setcounter{page}{1}

\section{Introduction}

\noindent The study about the interaction between geometric
properties of a projective variety $X \subset \P^n$ and structural
properties of its defining ideal $I(X)$ is one of the central issues
in algebraic geometry. For the last few decades, the problem of
giving conditions to guarantee that $I(X)$ is generated by quadrics
and the first few syzygy modules are generated by linear syzygies
have attracted considerable attention (cf. \cite{Gre1}, \cite{Gre2},
\cite{GL}, \cite{EL}, \cite{GP}, \cite{Ina}, etc). Another important
direction to study structures of defining equations of $X$ is to
examine whether $I(X)$ is defined by $2$-minors of one or several
linear determinantal presentations (cf. \cite{EKS}, \cite{Pu},
\cite{Ha}, \cite{B}, \cite{SS}, etc). In such a case, $I(X)$ admits
a generating set consisting of quadrics of rank $\leq 4$. The main
purpose of this paper is to exhibit many cases where $I(X)$ is
generated by quadrics of rank $3$. In particular, we prove that every Veronese embedding of $\P^n$ by
$\mathcal{O}_{\P^n} (d)$ has such a property unless ${\rm
char}~\Bbbk = 3$.

To state our results precisely, we begin with some notation and
definitions. Let $X \subset \P^n$ be any projective scheme over an algebraically closed field $\Bbbk$.
Through this paper, we assume that ${\rm char}~\Bbbk \neq 2$ since
the rank of a quadric is not well-defined over characteristic $2$
(e.g. $xy$). We say that \textit{$X \subset \P^n$ satisfies property
$\QR (k)$} if $I(X)$ is generated by quadrics of rank at most $k$.
For a very ample line bundle $L$ on a projective scheme $X$, we
will say that \textit{$L$ satisfies property $\QR (k)$} when so does
the linearly normal embedding $X \subset \P H^0 (X,L)$.

We can reinterpret property $\QR (k)$ as follows. Let $\P^N$,
$N={{n+2} \choose {2}}-1$, be the space of quadrics in $\P^n$ and
let $\Phi_k$, $1\leq k\leq n$, denote the variety of all quadrics of
rank at most $k$. Now, consider the subspace $\P (I(X)_2 )$ of
$\P^N$ and $\Phi_k (X) := \Phi_k \cap \P (I(X)_2 )$ as a projective
algebraic set in $\P (I(X)_2 )$. In this framework, $X$ satisfies
property $\QR (k)$ if and only if $\Phi_k (X)$ is \textit{nondegenerate} in
$\P (I(X)_2 )$.

There are lots of examples of property $\QR (4)$ in the literature.
First, let $\mathcal{L}$ be a line bundle of degree $d$ on a smooth
curve $\mathcal{C}$ of genus $g$. In \cite{SD}, B. Saint-Donat
proved that if $d \geq 2g+2$ then $\mathcal{C} \subset \P H^0
(\mathcal{C}, \mathcal{L})$ satisfies property $\QR (4)$. When
$\mathcal{C}$ is non-hyperelliptic, non-trigonal and not isomorphic
to a plane quintic, M. Green\cite{Gre} reproved the classical
Torelli's Theorem by showing that the canonical embedding
$\phi_{K_{\mathcal{C}}} ( \mathcal{C}) \subset \P^{g-1}$ satisfies
property $\QR (4)$. See also \cite{Pe, SD2, ArH}. Many classical
constructions in projective geometry such as rational normal
scrolls, Veronese varieties and Segre varieties of two projective
spaces are \textit{determinantally presented} in the sense that
their homogeneous ideals are generated by $2$-minors of a
$1$-generic matrix of linear forms (see e.g. \cite{EKS, Harris, Pu,
Ha}). Furthermore, any Segre-Veronese variety is defined
ideal-theoretically by $2$-minors of several linear determinantal
presentations, which are also called 'flattenings' (see \cite{B}).
Recently, Sidman and Smith in \cite{SS} proved that every
sufficiently ample line bundle on a projective connected scheme is
determinantally presented. So, they all satisfy property $\QR
(4)$.\\

Our first main theorem is about the property $\QR (3)$ of the
Veronese variety $V_{n,d} := \nu_d (\P^n ) \subset \P^N$, $N={{n+d}
\choose {n}}-1$, which is unexpected(!).

\begin{theorem}\label{thm:Veronese1}
Let $d, n$ be any positive integers and suppose that ${\rm
char}~\Bbbk \neq 2,3$. Then $(\P^n , \mathcal{O}_{\P^n} (d))$
satisfies property $\QR (3)$.
\end{theorem}

For the proof of this result, see Theorem \ref{thm:finite generating set of rank 3 quadrics}.

To prove Theorem \ref{thm:Veronese1}, the first step will be to find
many quadratic equations of rank $3$ in $I(V_{n,d})$. To this aim,
we use two methods. To explain the first one, let $f:H^0(\P^n
,\mathcal{O}_{\P^n} (d)) \rightarrow H^0(\P^N ,\mathcal{O}_{\P^N
}(1))$ be the natural isomorphism. Then we obtain the following map
\begin{equation*}
Q : H^0 (\P^n ,\mathcal{O}_{\P^n} (1)) \times H^0 (\P^n
,\mathcal{O}_{\P^n} (1)) \times H^0 (\P^n ,\mathcal{O}_{\P^n} (d-2))
\rightarrow I(V_{n,d})_2
\end{equation*}
defined by
\begin{equation*}
Q(s,t,h) = f(s \otimes s \otimes h) f(t \otimes t \otimes h) - f(s
\otimes t \otimes h )^2 .
\end{equation*}
This map is well-defined since $Q(s,t,h)$ is either $0$ or else a
rank $3$ quadratic equation of $V_{n,d}$. The second method is the
use of the natural group action of $\PGL_n (\Bbbk)$ on
$I(V_{n,d})_2$, by which the rank is preserved. That is, any $Q \in
I(V_{n,d})_2$ of rank $3$ and $\sigma \in \PGL_n (\Bbbk)$ give us a
new element $\sigma (Q) \in I(V_{n,d})_2$ of rank 3.

We prove Theorem \ref{thm:Veronese1} by using the
double induction on $n \geq 1$ and $d \geq 2$. So, we first prove
our theorem for the cases where $n=1$ (Corollary \ref{cor:rational
normal curve}) and $d=2$ (Theorem \ref{thm:secondVeronese}),
respectively. Then we complete the proof by combining the induction
hypothesis and the above two methods. In fact, this step is not
simple and we must deal with more than $10$ partial cases.

Let $W$ be the subspace of $I(V_{n,d})$ spanned by the image of the
above $Q$-map. In Theorem \ref{thm:producingR3Qs}, we provide an
explicitly defined finite subset $\Gamma$ of the image of the
$Q$-map which spans $W$. See Notation and Remarks
\ref{nar:producingQofR3} for the definition of $\Gamma$. In the
proof of Theorem \ref{thm:Veronese1}, it is shown that $I(V_{n,d})$
is generated by $\Gamma$. In particular, the property $\QR (3)$ of
$(\P^n , \mathcal{O}_{\P^n} (d))$ is induced by the decomposition
$\mathcal{O}_{\P^n} (d) = \mathcal{O}_{\P^n} (1)^2 \otimes \mathcal{O}_{\P^n} (d-2)$. \\

Our second main result is

\begin{theorem}\label{thm:Veronese2}
Suppose that ${\rm char}~\Bbbk = 3$. Then

\renewcommand{\descriptionlabel}[1]%
             {\hspace{\labelsep}\textrm{#1}}
\begin{description}
\setlength{\labelwidth}{13mm} \setlength{\labelsep}{1.5mm}
\setlength{\itemindent}{0mm}
\item[{\rm (1)}] $(\P^1 , \mathcal{O}_{\P^1} (d))$ satisfies property $\QR (3)$ for all $d \geq 2$.

\item[{\rm (2)}] $(\P^n , \mathcal{O}_{\P^n} (2))$ satisfies property $\QR (3)$ if and only if $n \leq 2$.
\end{description}
\end{theorem}

For the proof of this result, see Theorem \ref{thm:secondVeronese}.

For the cases $n=1$ and $n=d=2$, our proof of Theorem
\ref{thm:Veronese1} is indeed characteristic free. To prove the
failure of property $\QR (3)$ of $(\P^n , \mathcal{O}_{\P^n} (2))$
for $n \geq 3$, a crucial point is that any quadratic equation of
$\nu_2 (\P^n )$ of rank $3$ is obtained from the above $Q$-map.
Then, for $n=3$ we find a subset $\Gamma$ of $I (\nu_2 (\P^n ))$
with $|\Gamma|=19$ which generates the subspace spanned by the image
of the $Q$-map. Since $I (\nu_2 (\P^n ))_2$ is of $20$-dimension,
this shows that $(\P^3 , \mathcal{O}_{\P^3} (2))$ fails to satisfy
property $\QR (3)$. For $n \geq 4$, the proof comes from the fact
that $\nu_2 (\P^n )$ contains $\nu_2 (\P^3 )$ as an \textit{ideal-theoretic}
linear section (see Remark \ref{rmk:ideal version of Mumford's observation}).

When ${\rm char}~\Bbbk = 3$, $n \geq 2$ and $d \geq 3$, we do not
know yet whether $(\P^n , \mathcal{O}_{\P^n} (d))$ satisfies
property $\QR (3)$ or not. \\

Our third main result is about the asymptotic nature of the rank of
quadratic equations of the Veronese re-embedding of $(X,L)$ when $X$
is an arbitrary projective scheme and $L$ is a very ample line
bundle on $X$. In this direction, the first result is due to D.
Mumford, who proved in \cite[Theorem 1]{M} that if $X \subset \P^n$
is a nondegenerate irreducible projective variety, then
\begin{enumerate}
\item[$(i)$] the $d$th Veronese re-embedding of $X$ is
set-theoretically defined by quadrics, and
\item[$(ii)$] those quadrics can
be chosen as quadrics of rank at most $4$
\end{enumerate}
for all $d \geq \deg (X)$. Since \cite{M} had appeared, there have
been several interesting generalizations. In \cite{Gre1} and
\cite{Gre2}, Green considered the quadratic generation of the
homogeneous ideal as the first step towards understanding higher
syzygies. So far, numerous results have been reported in this
direction. Due to Green-Lazarsfeld in \cite{GL}, we say that $L$
satisfied condition $\mathrm{N}_p$ for some $p \geq 1$ if $X \subset
\P H^0 (X,L)$ is projectively normal and ideal-theoretically cut out
by quadrics such that the first $(p-1)$ steps of the minimal free
resolution of $I(X)$ are linear. In \cite{EL}, Ein and Lazarsfeld
proved that if $X$ is a complex smooth variety and $L$ is a very
ample line bundle on $X$ of degree $d_0$ then $L^d$ satisfies
condition $\mathrm{N}_{d+1-d_0}$ (see also \cite{GP}). Thus this
result generalizes the statement $(i)$ above. Also the statement
$(ii)$ is widely extended by many results on the determinantal
presentation of projective varieties, as mentioned above.

Along this line, our third main result is

\begin{theorem}\label{thm:main1}
Suppose that ${\rm char}~\Bbbk \neq 2,3$ and let $L$ be a very ample
line bundle on a projective scheme $X$ defining the linearly normal
embedding
\begin{equation*}
X \subset \P H^0 (X,L).
\end{equation*}
If $m$ is an integer such that $X$ is $j$-normal for all $j \geq m$ and $I(X)$ is generated by forms of
degree $\leq m$, then $(X,L^d )$ satisfies property $\QR (3)$ for all $d
\geq m$.
\end{theorem}

For the proof of this result, see the beginning of $\S~ \ref{sect5}$.

In Theorem \ref{thm:main1}, we can take $m$ to be the regularity of
$X \subset \P H^0 (X,L)$ in the sense of Castelnuovo-Mumford.

As an immediate application of Theorem \ref{thm:main1}, suppose that
$(X,L)$ satisfies Green-Lazarsfeld's condition $\mathrm{N}_1$. Then one could take $m=2$ and hence $(X,L^d )$ satisfies property $\QR (3)$ for all $d
\geq 2$ (see Corollary \ref{cor:Application 1}). This can be applied
to the Grassmannian manifolds (see Example \ref{ex:Grassmannian}).
In $\S$ \ref{sect5}, we provide a few examples which illustrate how to apply
Theorem \ref{thm:main1} to specific varieties.

\begin{remark}\label{rem:positivity and QR(3)}
Let $X$ be a projective scheme. Our main results in this paper show
that there is a significant correlation between some positive nature
of the very ample line bundle $L$ on $X$ and the property $\QR(3)$
of $(X,L)$, just like in many works on Green-Lazarsfeld's condition
$\mathrm{N}_p$ and Eisenbud-Koh-Stillman's determinantal
presentation of very ample line bundles. We will discuss more on
this direction in $\S~\ref{sect_eg_prbm}$.
\end{remark}

The paper is structured as follows. In $\S$ \ref{sect2}, we develop a method
to generate rank $3$ quadratic equations of $X \subset \P H^0
(X,L)$. Also we prove that every rational normal curve satisfies
property $\QR(3)$. In $\S$ \ref{sect3}, we study about property $\QR(3)$ of
the second Veronese embedding $(\P^n , \mathcal{O}_{\P^n} (2))$. In
$\S$ \ref{sect4}, we prove that every Veronese embedding $(\P^n ,
\mathcal{O}_{\P^n} (d))$ satisfies property $\QR(3)$ unless ${\rm
char}~\Bbbk = 2,3$. In $\S$ \ref{sect5}, we study the asymptotic behavior of
the property $\QR(3)$ of $(X,L^d )$ when $X$ is an arbitrary
projective scheme and $L$ is a very ample line bundle on $X$. We
also provide some applications of the results to the case of complete
intersections, Grassmannian manifolds, Abelian varieties, Enriques
surfaces and K3 surfaces. We finish the paper by giving relevant
examples, some problems for further direction in $\S$ \ref{sect_eg_prbm}.\\

\noindent {\bf Acknowledgement.} We would like to thank Daeyeol Jeon
for Example \ref{ex:canonical curve}. Kangjin Han was supported by
the National Research Foundation of Korea (NRF) grant funded by the
Korea government (Ministry of Science and ICT, no.
2017R1E1A1A03070765) and the DGIST Start-up Fund (Ministry of
Science, ICT and Future Planning, no. 2016010066). Wanseok Lee was
supported by Basic Science Research Program through NRF funded by
the Ministry of Education (no. 2020R1F1A1A01049999). Hyunsuk Moon was supported by the National Research Foundation of Korea(NRF) grant funded by the Korea government(MSIT) (no.2017R1E1A1A03070765). Euisung Park
was supported by the Korea Research Foundation Grant funded by the
Korean Government (no. 2018R1D1A1B07041336). The computer algebra
software \texttt{Macaulay2} \cite{GS} was very useful for many
related experiments and expectations. We are grateful to the anonymous referee for careful reading and valuable suggestions

\section{Rank $3$ quadratic generators}\label{sect2}
\noindent This section is devoted to introduce our method to
generate rank $3$ quadrics in the ideal of a projective scheme.

\begin{notation and remarks}\label{nar:producingQofR3}
Let $X$ be a projective scheme and $L$ a very ample line bundle on
$X$. Suppose that $L$ is decomposed as $L=L_1^{\otimes 2}\otimes
L_2$ where $L_1$ and $L_2$ are line bundles on $X$ such that
\begin{equation*}
p := h^0 (X,L_1 ) \geq 2 \quad \mbox{and} \quad q := h^0 (X,L_2 )
\geq 1.
\end{equation*}

\renewcommand{\descriptionlabel}[1]%
             {\hspace{\labelsep}\textrm{#1}}
\begin{description}
\setlength{\labelwidth}{13mm} \setlength{\labelsep}{1.5mm}
\setlength{\itemindent}{0mm}

\item[{\rm (1)}] The linearly normal embedding $X\subset \P
H^0(X,L)=\P^r$ induces an isomorphism
\begin{equation*}
f:H^0(X,L) \rightarrow  H^0(\P^r,\mathcal{O}_{\P^r}(1)).
\end{equation*}

\item[{\rm (2)}] Define the map
\begin{equation*}
Q = Q_{L_1 , L_2} : H^0 (X,L_1 ) \times H^0 (X,L_1 ) \times H^0
(X,L_2 ) \rightarrow I(X)_2
\end{equation*}
by
\begin{equation*}
Q(s,t,h) = f(s \otimes s \otimes h) f(t \otimes t \otimes h) - f(s
\otimes t \otimes h )^2 .
\end{equation*}
This map is well-defined since the restriction of $Q(s,t,h)$ to $X$
becomes
\begin{equation*}
(s\otimes s\otimes h)|_X \times(t\otimes t\otimes h)|_X - (s\otimes
t\otimes h)|_X ^2 =0
\end{equation*}
(cf. \cite[Proposition 6.10]{E}). Thus, $Q(s,t,h)$ is either
$0$ or else a rank $3$ quadratic equation of $X$ (see also Lemma \ref{exchange}).

\item[{\rm (3)}] (\textit{A sufficient way to guarantee $\QR(3)$}) Let $\mathfrak{Q} = \mathfrak{Q}_{L_1 ,L_2}$ be the ideal generated by the image of the map $Q_{L_1 , L_2 }$. Thus $\mathfrak{Q} \subseteq I(X)$ and $(X,L)$ satisfies property $\QR(3)$ if the equality $\mathfrak{Q}= I(X)$ is attained.

\item[{\rm (4)}] (\textit{Featured subsets of $W_{L_1 ,L_2}$}) Let $\{s_1,s_2,\ldots, s_p\}$ and  $\{h_1,h_2,\ldots,h_q\}$ be any chosen bases for $H^0
(X,L_1 )$ and $H^0 (X,L_2 )$ respectively. Also, let
$$\begin{cases}
\Delta_1 := \{ (i,j)~|~ i,j \in \{1,\ldots ,p \},~i<j \},\\
\Delta_2 := \{ (i,j,k) ~|~ (i,j) \in \Delta_1 ,~k \in \{1,\ldots ,p \},~k \neq i,j \},\\
\Delta_3 := \{ ( i,j,k, l)) ~|~ (i,j),(k, l) \in \Delta_1 ,~i<k ,~ j \neq k,l  \}, \quad \mbox{and} \\
H := \{ h_1 , \ldots , h_q \} \cup \{ h_i + h_j ~|~ 1 \le i < j \le q \}.
\end{cases}$$
Concerned with the problem of finding a spanning set of $W=W_{L_1 ,L_2}$, the subspace of $I(X)_2$
spanned by the image of the map $Q_{L_1 , L_2 }$, we define three types of finite subsets of $W$ as follows:
$$\begin{cases}
\Gamma_{11}=\big\{Q(s_i,s_j,h) ~|~ (i,j) \in \Delta_1 , ~h \in H \big\}\\
\Gamma_{12}=\big\{Q(s_i + s_j , s_k,h)~|~ (i,j,k) \in \Delta_2 , ~h \in H  \big\}\\
\Gamma_{22}=\big\{Q(s_i+s_j,s_k +s_l,h )~|~ (i,j,k,l)  \in \Delta_3 , ~h \in H \big\}
\end{cases}$$
Let $\Gamma (L_1 , L_2)$ be the union $\Gamma_{11}\cup \Gamma_{12}
\cup \Gamma_{22}$. Then it can be checked that
\begin{equation}\label{eq:upper bound}
|\Gamma (L_1 , L_2)| \leq {{{{p} \choose {2}}+1} \choose 2} \times
{q+1 \choose 2}.
\end{equation}
\end{description}
\end{notation and remarks}

Our main result in this section is

\begin{theorem}\label{thm:producingR3Qs}
Keep the notations in Notation and Remarks \ref{nar:producingQofR3}.
Then,
\begin{equation*}
\textrm{the subspace $W_{L_1 , L_2 }$ is spanned by $\Gamma(L_1 ,
L_2)$ as a $\Bbbk$-vector space.}
\end{equation*}
\end{theorem}

To prove Theorem \ref{thm:producingR3Qs}, we begin with the
following

\begin{lemma}\label{exchange}
Keep the notations in Notation and Remarks \ref{nar:producingQofR3}.
Then

\renewcommand{\descriptionlabel}[1]%
             {\hspace{\labelsep}\textrm{#1}}
\begin{description}
\setlength{\labelwidth}{13mm} \setlength{\labelsep}{1.5mm}
\setlength{\itemindent}{0mm}

\item[{\rm (1)}] For any $s,t\in H^0(X,L_1)$, $h\in H^0(X,L_2)$ and a constant $\lambda\in\kk$, it holds that
\begin{equation*}
Q(s,s,h)=0, Q(s,t,h)=Q(t,s,h),Q(s,s+t,h)=Q(s,t,h)
\end{equation*}
\begin{equation*}
\text{~and~} Q(\lambda s,t,h)=Q(s,t,\lambda h)=\lambda^2 Q(s,t,h).
\end{equation*}

\item[{\rm (2)}] For $s,t,u\in H^0(X,L_1)$, $g,h \in H^0(X,L_2)$ and $a,b \in \Bbbk$, it holds that
\begin{equation*}
Q(s,at+bu,h) = (a^2 -ab) Q(s,t,h) + (b^2 -ab) Q(s,u,h) + ab Q(s,t+u,h)
\end{equation*}
and
\begin{equation*}
Q(s,t,ag+bh) = (a^2 -ab) Q(s,t,g) + (b^2 -ab) Q(s,t,h) + ab Q(s,t,g+h).
\end{equation*}

\item[{\rm (3)}] For $m\geq 3$, let $s,t,t_1 , \ldots ,
t_m \in H^0(X,L_1)$ and $h,g_1 , \ldots , g_m \in H^0(X,L_2)$. Then
\begin{equation*}
Q(s,t_1+t_2+\cdots+t_m,h)=\sum_{1\leq i<j\leq m} Q(s,t_i+t_j,h)
-(m-2)\sum_{i=1}^m Q(s,t_i,h)
\end{equation*}
and
\begin{equation*}
Q(s,t,g_1+g_2+\cdots+g_m)=\sum_{1\leq i<j\leq m} Q(s,t,g_i+g_j)
-(m-2)\sum_{i=1}^m Q(s,t,g_i).
\end{equation*}

\item[{\rm (4)}] For $s,t,u\in H^0(X,L_1)$ and $h \in H^0(X,L_2)$, it holds that
\begin{equation*}
Q(s+u,t+u,h) = Q(s,u,h)+Q(t,u,h)+Q(s+u,t,h)+Q(t+u,s,h)-Q(s,t,h) -Q(s+t,u,h).
\end{equation*}
\end{description}
\end{lemma}

\begin{proof}
$(1)$ Note that $f(s\otimes t\otimes h)=f(t\otimes s\otimes h)$. Then the statement can be easily shown by using the definition of the map
$Q$.

\noindent $(2)$ One can check that
\begin{equation*}
Q(s,at+bu,h)= a^2 Q(s,t,h) +b^2 Q(s,u,h) +2ab P(s,t,u,h)
\end{equation*}
where $P(s,t,u,h) := f(s \otimes s \otimes h) f(t \otimes u \otimes h)-f(s
\otimes t \otimes h)f(s \otimes u \otimes h)$. In particular, it holds that
\begin{equation*}
Q(s,t+u,h)= Q(s,t,h) + Q(s,u,h) +2P(s,t,u,h).
\end{equation*}
Thus we have
\begin{equation*}
Q(s,at+bu,h)= a^2 Q(s,t,h) +b^2 Q(s,u,h) +ab \{ Q(s,t+u,h)- Q(s,t,h) - Q(s,u,h)\},
\end{equation*}
which shows the first formula. Similarly, one can prove that
\begin{equation*}
Q(s,t,ag+bh) = a^2 Q(s,t,g) + b^2 Q (s,t,h)+ab R(s,t,g,h)
\end{equation*}
where $R(s,t,g,h) := f(s \otimes s \otimes g)f(t \otimes t \otimes
h)+ f(s \otimes s \otimes h)f(t \otimes t \otimes g)-2f(s \otimes t
\otimes g)f(s \otimes t \otimes h)$. Then it holds that
\begin{equation*}
Q(s,t, g+ h) =   Q(s,t,g) +   Q (s,t,h)+ R(s,t,g,h)
\end{equation*}
and hence we get
\begin{equation*}
Q(s,t,ag+bh)= a^2 Q(s,t,g) +b^2 Q(s,t,h) +ab \{ Q(s,t, g+ h)-  Q(s,t,g)- Q (s,t,h) \},
\end{equation*}
which shows the second formula.

\noindent $(3)$ We will prove the first equality by using induction
on $m \geq 3$. The second one can be shown by a similar way.
For simplicity, let $g(s,t):=f(s\otimes t\otimes h)$ in this proof.
For $m=3$,
\begin{align*}
 Q(s,t_1+t_2+t_3,h)~ = ~&g(s,s)g(t_1+t_2+t_3,t_1+t_2+t_3)-g(s,t_1+t_2+t_3)^2\\
 ~ = ~&g(s,s)\big\{g(t_1,t_1)+g(t_2,t_2)+g(t_3,t_3)+2g(t_1,t_2)+2g(t_1,t_3)\\
 ~~&+2g(t_2,t_3)\big\}-\big\{g(s,t_1)^2+g(s,t_2)^2+g(s,t_3)^2+2g(s,t_1)g(s,t_2)\\
 ~~&+2g(s,t_1)g(s,t_3)+2g(s,t_2)g(s,t_3)\big\}.
\end{align*}
Since for $1\leq i<j\leq 3$
\begin{align*}
Q(s,t_i+t_j,h)~ = ~&g(s,s)g(t_i+t_j,t_i+t_j)-g(s,t_i+t_j)^2 \\
~ = ~&g(s,s)\big\{g(t_i,t_i)+g(t_j,t_j)+2g(t_i,t_j)\big\}-\big\{g(s,t_i)^2+g(s,t_j)^2+2g(s,t_i)g(s,t_j)\big\},
\end{align*}
we can calculate the difference
\begin{align*}
&Q(s,t_1+t_2,h)+Q(s,t_1+t_3,h)+Q(s,t_2+t_3,h)-Q(s,t_1+t_2+t_3,h)  \\
&~\quad = ~ g(s,s)\big\{g(t_1,t_1)+g(t_2,t_2)+g(t_3,t_3)\big\}-\big\{g(s,t_1)^2+g(s,t_2)^2+g(s,t_3)^2\big\}\\
&~\quad = ~ Q(s,t_1,h)+Q(s,t_2,h)+Q(s,t_3,h)
\end{align*}
Hence, for $m=3$, the statement does hold.

For $m \geq 4$, now let us denote $Q(s,t,h)$ by $F(t)$ for the simplicity. By induction hypothesis, we have
\begin{align*}
F(t_1 + \cdots + t_m ) ~ = ~& F(t_1 + \cdots + (t_{m-1} + t_m ) ) \\
~ = ~& \sum_{1\leq i<j\leq m-2} F(t_i+t_j) + \sum_{i=1}^{m-2} F(t_i+t_{m-1}+t_m)\\
~ ~&-(m-3)\Big\{\sum_{i=1}^{m-2}
F(t_i) +F(t_{m-1}+t_m)\Big\}
\end{align*}

and
\begin{equation*}
F(t_i+t_{m-1}+t_m) = F(t_i + t_{m-1} ) + F(t_i + t_m ) + F(t_{m-1}
+t_m ) - F(t_i ) - F (t_{m-1} ) - F(t_m )
\end{equation*}
for each $1 \leq i \leq m-2$. Hence, by combining these identities,
we get the desired formula for $Q(s,t_1+t_2+\cdots+t_m,h)$.

\noindent $(4)$ From Lemma \ref{exchange}.(1) and (3), we have
\begin{align*}
Q(s+t,u,h) &~ = ~ Q(s+t,s+t+u,h)  \\
&~ = ~  Q(s+t,s+u,h)+Q(s+t,t+u,h)-2Q(s,t,h) -Q(s+t,u,h)
\end{align*}
and hence
\begin{equation}\label{eq:a common term}
Q(s+t,s+u,h)+Q(s+t,t+u,h) =2Q(s,t,h)+2Q(s+t,u,h).
\end{equation}
By permuting $s,t$ and $u$ in (\ref{eq:a common term}), it follows that
\begin{equation}\label{eq:a corollary 1}
Q(s+t,s+u,h)+Q(s+u,t+u,h)=2Q(s,u,h)+2Q(s+u,t,h)
\end{equation}
and
\begin{equation}\label{eq:a corollary 2}
Q(s+t,t+u,h)+Q(s+u,t+u,h)=2Q(t,u,h)+2Q(t+u,s,h).
\end{equation}
Now, the desired identity comes from (\ref{eq:a common term}), (\ref{eq:a corollary 1}) and (\ref{eq:a corollary 2}).
\end{proof}

Now we are ready to give a \\

\noindent {\bf Proof of Theorem \ref{thm:producingR3Qs}.} Consider
the general member
\begin{equation*}
G := Q (a_1 s_1 + \cdots + a_p s_p , b_1 s_1 + \cdots + b_p s_p ,
c_1 h_1 + \cdots + c_q h_q )
\end{equation*}
of the image of the map $Q_{L_1 ,L_2}$ where $a_1 ,\ldots , a_p ,
b_1 , \ldots , b_p , c_1 , \ldots , c_q$ are from $\Bbbk$. By
applying Lemma \ref{exchange}.(1) and (3) repeatedly, one can see
that $G$ is a $\Bbbk$-linear combination of the quadratic equations of the form
\begin{equation*}
Q( \alpha_1 s_i +  \beta_1 s_j , \alpha_2 s_k +  \beta_2 s_l ,  \alpha_3 h_m +  \beta_3 h_n )
\end{equation*}
where $\alpha_u , \beta_v \in \Bbbk$. Then, by Lemma \ref{exchange}.(2), we may assume that $\alpha_u , \beta_v \in \{ 0,1 \}$.
Finally, Lemma \ref{exchange}.(4) enables us to exclude the cases where $\{i,j \} \cap \{k,l \}$ is not empty.  \qed \\

We finish this section by applying Theorem \ref{thm:producingR3Qs}
to the rational normal curves.

\begin{corollary}\label{cor:rational normal curve}
For every $d \geq 2$, let $\mathcal{C}_d \subset \P^d$ be the
standard rational normal curve of degree $d$. Then
$I(\mathcal{C}_d)$ is generated by $\Gamma (\mathcal{O}_{\P^1}
(1),\mathcal{O}_{\P^1} (d-2))$. In particular, every rational normal
curve satisfies property $\QR(3)$.
\end{corollary}

\begin{proof}
Consider the $Q$-map about the decomposition
\begin{equation*}
\mathcal{O}_{\P^1} (d) = \mathcal{O}_{\P^1} (1)^{\otimes 2} \otimes
\mathcal{O}_{\P^1} (d-2).
\end{equation*}
Let $\{s,t \}$ be a basis for $H^0 (\P^1 ,\mathcal{O}_{\P^1} (1))$.
So, $\{ s^{d-2-i} t^i ~|~ 0 \leq i \leq d-2 \}$ is a basis for $H^0
(\P^1 ,\mathcal{O}_{\P^1} (d-2))$. Then $\Gamma := \Gamma
(\mathcal{O}_{\P^1} (1) ,\mathcal{O}_{\P^1} (d-2))$ is equal to
\begin{equation*}
\big\{ Q(s ,t , s^{d-2-i} t^i) ~|~ 0 \leq i \leq d-2 \big\} \bigcup
\big\{ Q(s,t,s^{d-2-i} t^i + s^{d-2-j} t^j) ~|~ 0 \leq i<j\leq d-2
\big\}.
\end{equation*}
Note that $\Gamma
(\mathcal{O}_{\P^1} (1) ,\mathcal{O}_{\P^1})=\{ Q(s,t,1)=f(s^2)f(t^2)-f(st)^2=z_0 z_2 - z_1^2 \}$.
Now, we will show that $\Gamma$ has exactly $\binom{d}{2}$ quadrics
of rank $3$ and that they are $\Bbbk$-linearly independent. Let $z_0
, z_1 , \ldots , z_d$ be the homogeneous coordinates of $\P^d$. Then
\begin{equation*}
F_i:=Q(s ,t , s^{d-2-i} t^i) = z_i z_{i+2} - z_{i+1} ^2
\end{equation*}
and
\begin{equation*}
G_{i,j}:=Q(s,t,s^{d-2-i} t^i + s^{d-2-j} t^j) = (z_i + z_j )
(z_{i+2} + z_{j+2} ) - (z_{i+1} + z_{j+1} )^2 .
\end{equation*}
Let
$G_{i,j}':=G_{i,j}-F_i-F_j=z_iz_{j+2}+z_jz_{i+2}-2z_{i+1}z_{j+1}$.
The leading terms of $F_i$'s are
$z_iz_{i+2}$ and those of $G_{i,j}'$'s are $z_i z_{j+2}$ in the standard lexicographic order. They are all distinct, because $i<j$. It
means that all $F_i$'s and $G_{i,j}'$'s are $\kk$-linearly
independent and
\begin{equation*}
|\Gamma (\mathcal{O}_{\P^1} (1) ,\mathcal{O}_{\P^1} (d-2))| =
(d-1)+\binom{d-1}{2}=\binom{d}{2} = {\dim}_{\Bbbk}~I(\mathcal{C}_d
)_2 .
\end{equation*}
In consequence, it is shown that $\Gamma$ generates $I(\mathcal{C}_d
)$.
\end{proof}

\section{Second Veronese embeddings}\label{sect3}
\noindent This section is devoted to solving the problem whether the
second Veronese variety satisfies property $\QR(3)$ or not. It is
interesting that the answer for this question depends on the
characteristic of the base field $\Bbbk$.

\begin{theorem}\label{thm:secondVeronese}
Let $n$ be any positive integer. Then

\renewcommand{\descriptionlabel}[1]%
             {\hspace{\labelsep}\textrm{#1}}
\begin{description}
\setlength{\labelwidth}{13mm} \setlength{\labelsep}{1.5mm}
\setlength{\itemindent}{0mm}
\item[{\rm (1)}] Suppose that ${\rm char}~\Bbbk \neq 2,3$. Then $I(V_{n,2})$ is generated by
$\Gamma (\mathcal{O}_{\P^n} (1),\mathcal{O}_{\P^n})$. In particular,
$(\P^n , \mathcal{O}_{\P^n} (2))$ satisfies property $\QR(3)$.

\item[{\rm (2)}] If ${\rm char}~\Bbbk =3$, then $(\P^n , \mathcal{O}_{\P^n} (2))$ satisfies property
$\QR(3)$ if and only if $n \leq 2$.
\end{description}
\end{theorem}

We begin with fixing a few notation related to the Veronese
embedding of projective spaces, which we use throughout the
remaining part of this paper.

\begin{notation and remarks}\label{not and rmk:Veronese}
By $\N_{0}$ we denote the set of non-negative integers. For the
integers $n \geq 1$ and $d \geq 2$, consider the Veronese variety
\begin{equation*}
V_{n,d} := \nu_d (\P^n ) \subset \P^{N(n,d)} \quad \mbox{where}
\quad N(n,d) = {{n+d} \choose {n}} -1.
\end{equation*}
For $A(n,d) := \{ (a_0 , \ldots ,a_n )~|~ a_i \in \N_0 \quad \mbox{and}
\quad a_0 +   \cdots + a_n = d \}$, let
\begin{equation*}
B(n,d) := \{ z_I ~|~ I \in A(n,d) \}
\end{equation*}
be the set of standard homogeneous coordinates of the projective
space $\P^{N(n,d)}$.

\renewcommand{\descriptionlabel}[1]%
             {\hspace{\labelsep}\textrm{#1}}
\begin{description}
\setlength{\labelwidth}{13mm} \setlength{\labelsep}{1.5mm}
\setlength{\itemindent}{0mm}
\item[{\rm (1)}] The homogeneous ideal of
$V_{n,d}$ in $\P^{N(n,d)}$ is generated by the set of quadrics
\begin{equation*}
\mathcal{Q} (n,d) := \{ z_I z_J - z_K z_L ~|~ I,J,K,L \in  A(n,d)
\quad \mbox{and} \quad I+J = K+L \}.
\end{equation*}
In particular, $(\P^n , \mathcal{O}_{\P^n} (d))$ does always satisfy $\QR(4)$.

\item[{\rm (2)}] For $I =(a_0 , a_1 , \ldots ,a_n ) \in A(n,d)$ and $k \in
\{0,1,\ldots , n \}$, we will denote $a_k$ by $I_k$ and $\displaystyle\sum_{i=0}^n a_i$ by $|I|$.

\item[{\rm (3)}] For each $0 \leq k \leq n$, consider the inclusion map from
$A(n-1,d)$ into $A(n,d)$ defined by inserting $0$ to the $k$th
location. This map identifies $B(n-1,d)$ with a subset of $B(n,d)$.
Also this map induces an inclusion map
\begin{equation*}
\iota_k : I( V_{n-1,d} )_2 \rightarrow I(V_{n,d} )_2 .
\end{equation*}
Moreover, $\iota_k$ maps the subset $\Gamma (\mathcal{O}_{\P^{n-1}}
(1) ,\mathcal{O}_{\P^{n-1}} (d-2))$ into $\Gamma
(\mathcal{O}_{\P^{n}} (1) ,\mathcal{O}_{\P^{n}} (d-2))$.

\item[{\rm (4)}] Suppose that $d \geq 3$. For each $0 \leq k \leq n$, consider the inclusion map from
$A(n,d-1)$ into $A(n,d)$ defined by adding $1$ to the $k$th
location. This corresponds to the multiplication map by $x_i$ of the
homogeneous coordinate ring of $\P^n$. This map identifies
$B(n,d-1)$ with a subset of $B(n,d)$. Also this map induces the map
$$\delta_k : I(V_{n,d-1})_2\rightarrow I(V_{n,d})_2.$$
Moreover, $\delta_k$ maps the subset $\Gamma (\mathcal{O}_{\P^{n}}
(1) ,\mathcal{O}_{\P^{n}} (d-3))$ into $\Gamma (\mathcal{O}_{\P^{n}}
(1) ,\mathcal{O}_{\P^{n}} (d-2))$.
\end{description}
\end{notation and remarks}
\smallskip

The following proposition plays a crucial role in the proof of
Theorem \ref{thm:secondVeronese}.

\begin{proposition}\label{prop:Veronese lower d}
If $I(V_{2d-1 ,d})$ is generated by $\Gamma (\mathcal{O}_{\P^{2d-1}}
(1),\mathcal{O}_{\P^{2d-1}} (d-2))$, then $I(V_{n ,d})$ is generated
by $\Gamma (\mathcal{O}_{\P^{n}} (1),\mathcal{O}_{\P^{n}} (d-2))$
for all $n \geq 2d-1$.
\end{proposition}

\begin{proof}
We use induction on $n \geq 2d-1$. The case $n = 2d-1$ is done by our assumption.

Suppose that $n \geq 2d$ and let $Q = z_I z_J-z_K z_L$ be an element
of $\mathcal{Q} (n,d)$. We need to show that $Q$ is a linear
combination of quadratic equations in $\Gamma (\mathcal{O}_{\P^{n}}
(1),\mathcal{O}_{\P^{n}} (d-2))$. Since $n \geq 2d$, by pigeonhole
principle, there exists $0\leq k\leq n$ such that $I_k = J_k = K_k =
L_k =0$. Now, by using the inclusion map
\begin{equation*}
\iota_k : I(V_{n-1,d} )_2 \rightarrow I(V_{n,d} )_2 ,
\end{equation*}
we can regard $Q$ as an element of $I( V_{n-1 , d})$. By induction
hypothesis, $Q$ is a linear combination of elements in $\Gamma
(\mathcal{O}_{\P^{n-1}} (1),\mathcal{O}_{\P^{n-1}} (d-2))$. Then, by
Notation and Remarks \ref{not and rmk:Veronese}.(3), $Q$ is also a linear
combination of elements in $\Gamma (\mathcal{O}_{\P^{n}}
(1),\mathcal{O}_{\P^{n}} (d-2))$. This completes the proof.
\end{proof}
\smallskip

Now, we give the\\

\noindent {\bf Proof of Theorem \ref{thm:secondVeronese}.} $(1)$ By
Proposition \ref{prop:Veronese lower d}, it is enough to check the
statement for the cases where $n=2$ and $n=3$. So, consider the
$Q$-map about the decomposition
\begin{equation*}
\mathcal{O}_{\P^n}(2)= \mathcal{O}_{\P^n}(1)^{\otimes 2}\otimes
\mathcal{O}_{\P^n}.
\end{equation*}
Let $\{ x_0,\ldots,x_n \}$ and $\{ 1 \}$ be respectively bases of
$H^0(\P^n , \mathcal{O}_{\P^n}(1))$ and $H^0(\P^n ,
\mathcal{O}_{\P^n})$. Thus the $\Bbbk$-vector space
$W_{\mathcal{O}_{\P^n}(1),\mathcal{O}_{\P^n}}$ defined in Notation
and Remarks \ref{nar:producingQofR3}.(3) is generated by the union
$\Gamma:=\Gamma (\mathcal{O}_{\P^{n}} (1),\mathcal{O}_{\P^{n}})= \Gamma_{11}
\cup \Gamma_{12} \cup \Gamma_{22}$ where
$$\begin{cases}
\Gamma_{11}=\big\{F_{i,j} := Q(x_i,x_j,1) ~|~ (i,j) \in \Delta_1  \big\},\\
\Gamma_{12}=\big\{G_{i,j,k} := Q(x_i , x_j + x_k,1)~|~ (i,j,k) \in \Delta_2   \big\} \quad \mbox{and}\\
\Gamma_{22}=\big\{H_{i,j,k,l} := Q(x_i +x_j,x_k +x_l,1 )~|~ (i,j,k,l)  \in \Delta_3  \big\}.
\end{cases}$$

For $n=2$, $\Gamma (\mathcal{O}_{\P^{2}} (1),\mathcal{O}_{\P^{2}})$
consists of the following six elements:
$$\begin{cases}
F_{0,1} := Q(x_0,x_1,1)=z_{200}z_{020}-z_{110}^2 \\
F_{0,2} := Q(x_0,x_2,1)=z_{200}z_{002}-z_{101}^2 \\
F_{1,2} := Q(x_1,x_2,1)=z_{020}z_{002}-z_{011}^2 \\
G_{0,1,2} := Q(x_0,x_1+x_2,1)=z_{200}(z_{020}+2z_{011}+z_{002})-(z_{110}+z_{101})^2 \\
G_{1,0,2} := Q(x_1,x_0+x_2,1)=z_{020}(z_{200}+2z_{101}+z_{002})-(z_{110}+z_{011})^2 \\
G_{2,0,1} := Q(x_2,x_0+x_1,1)=z_{002}(z_{200}+2z_{110}+z_{020})-(z_{101}+z_{011})^2
\end{cases}$$
One can easily check that the following three identities hold:
$$\begin{cases}
G_{0,1,2} = F_{0,1} + F_{0,2} +2(z_{200}z_{011}-z_{110}z_{101})\\
G_{1,0,2} = F_{0,1} + F_{1,2} +2(z_{020}z_{101}-z_{110}z_{011})\\
G_{2,0,1} = F_{0,2} + F_{1,2} +2(z_{002}z_{110}-z_{101}z_{011})
\end{cases}$$
Since ${\rm char}~\Bbbk \neq 2$, this verifies that the homogeneous ideal
\begin{equation*}
I(V_{2,2}) = \langle F_{0,1} , F_{0,2} , F_{1,2} ,
z_{200}z_{011}-z_{110}z_{101} , z_{020}z_{101}-z_{110}z_{011} , z_{002}z_{110}-z_{101}z_{011} \rangle
\end{equation*}
is generated by $\Gamma (\mathcal{O}_{\P^{2}}
(1),\mathcal{O}_{\P^{2}})$.

For $n=3$, note that $\dim_{\Bbbk}~I(V_{3,2})=20$ and $I(V_{3,2})$ is generated by the $2$-minors of
the $4\times 4$ symmetric matrix
$$A=\begin{pmatrix}
z_{2000} & z_{1100} & z_{1010} & z_{1001}\\
z_{1100} & z_{0200} & z_{0110} & z_{0101}\\
z_{1010} & z_{0110} & z_{0020} & z_{0011}\\
z_{1001} & z_{0101} & z_{0011} & z_{0002}
\end{pmatrix}.$$
For each $i \in \{1,2,3,4 \}$, let $A_i$ be the principal submatrix
of $A$ obtained by eliminating the $i$th row and column. Then
$2$-minors of $A_i$ are contained in $\iota_i (I(V_{2,2})_2 )$. So
they can be handled by the previous case. The remaining generators
not coming from these principal submatrices are
\begin{equation*}
R_1=z_{1100}z_{0011}-z_{1010}z_{0101}, ~R_2=z_{1100}z_{0011}-z_{1001}z_{0110} ~\mbox{and} ~
R_3=z_{1010}z_{0101}-z_{1001}z_{0110}.
\end{equation*}
Note that $R_3=-R_1 +R_2$. Thus it suffices to show that $R_1$ and $R_2$ are $\Bbbk$-linear
combinations of elements in $\Gamma=\Gamma (\mathcal{O}_{\P^{3}} (1),\mathcal{O}_{\P^{3}})$. Now, consider the set
\begin{equation*}
\Gamma_{22} = \{H_{0,1,2,3},H_{0,2,1,3} , H_{0,3,1,2} \}.
\end{equation*}
Letting $H'_{0,1,2,3}=4z_{1100}z_{0011}-2z_{1010}z_{0101}-2z_{1001}z_{0110}$, we have
\begin{align*}
H_{0,1,2,3} =& ~
(z_{2000}+2z_{1100}+z_{0200})(z_{0020}+2z_{0011}+z_{0002})-(z_{1010}+z_{1001}+z_{0110}+z_{0101})^2
\\
=& ~G_{0,2,3}+G_{1,2,3}+G_{2,0,1}+G_{3,0,1}-F_{0,2}-F_{0,3}-F_{1,2}-F_{1,3} + H'_{0,1,2,3} .
\end{align*}
Then it holds that the elements $H'_{0,1,2,3}$ is contained in
$W_{\mathcal{O}_{\P^3}(1),\mathcal{O}_{\P^3}}$.
Similarly, one can show that two elements
\begin{equation*}
H'_{0,2,1,3}=4z_{1010}z_{0101}-2z_{1100}z_{0011}-2z_{1001}z_{0110}
\end{equation*}
and
\begin{equation*}
H'_{0,3,1,2}=4z_{1001}z_{0110}-2z_{1100}z_{0011}-2z_{1010}z_{0101}
\end{equation*}
are also contained in $W_{\mathcal{O}_{\P^3}(1),\mathcal{O}_{\P^3}}$. Since
\begin{equation}\label{eq:char 3 linear relation}
H'_{0,1,2,3}=2R_1+2R_2,~ H'_{0,2,1,3}=-4R_1+2R_2 \quad \mbox{and}
\quad H'_{0,3,1,2}=2R_1-4R_2
\end{equation}
and since ${\rm char}~\Bbbk \neq 2,3$, it follows that
\begin{equation*}
R_1=\frac{1}{6}(H'_{0,1,2,3}-H'_{0,2,1,3}) \quad \mbox{and} \quad
R_2=\frac{1}{6}(H'_{0,1,2,3}-H'_{0,3,1,2})
\end{equation*}
are contained $W_{\mathcal{O}_{\P^3}(1),\mathcal{O}_{\P^3}}$. This
completes the proof that the ideal $I (V_{3,2})$ is generated by
$\Gamma$ and hence $(\P^3, \mathcal{O}_{\P^3}(2))$ satisfies
property $\QR(3)$.

\noindent $(2)$ Suppose that ${\rm char}~\Bbbk = 3$. First note that
Corollary \ref{cor:rational normal curve} for $(\P^1,
\mathcal{O}_{\P^1}(2))$ and the proof in $(1)$ for $(\P^2,
\mathcal{O}_{\P^2}(2))$ are still valid in ${\rm char}~\Bbbk=3$.
Thus, it is enough to show that $(\P^n, \mathcal{O}_{\P^n}(2))$
fails to satisfy property $\QR(3)$ if $n \geq 3$.

When $n=3$, we will first show that $W := W_{\mathcal{O}_{\P^3}(1),\mathcal{O}_{\P^3}}$ is a proper subset of $I(V_{3,2})_2$. Indeed, $W  = \langle \Gamma \rangle$ by Theorem \ref{thm:producingR3Qs} and $\Gamma$ has at most $21$ elements by (\ref{eq:upper bound}). Furthermore,
\begin{equation*}
\Gamma_{22} = \{ H_{0,1,2,3} ,  H_{0,2,1,3} ,  H_{0,3,1,2} \}
\end{equation*}
and $\Gamma_{11} \cup \Gamma_{12}$ has at most $18$ elements. From (\ref{eq:char 3 linear relation}), we obtain
\begin{equation*}
H'_{0,1,2,3}=  H'_{0,2,1,3}=  H'_{0,3,1,2}.
\end{equation*}
since ${\rm char}~\Bbbk = 3$. In particular, the image of $\Gamma_{22}$ in the quotient space $W / \langle \Gamma_{11} \cup \Gamma_{12} \rangle$ has at most one nonzero element. In consequence, we have
\begin{equation*}
\dim_{\Bbbk} ~ W  \leq | \Gamma_{11} \cup \Gamma_{12} | +1 = 19.
\end{equation*}
This shows that $W$ is strictly smaller than $I(V_{3,2})_2$ since $\dim_{\Bbbk} ~I(V_{3,2})_2 = 20$. Next, let $W'$ be the $\Bbbk$-vector space spanned by all rank $3$ elements in $I(V_{3,2})_2$. We claim that $W' = W$ and hence any quadric of rank $3$ in $I(V_{3,2})$ is contained in $W$. Obviously, this completes the proof. That is, $(\P^3, \mathcal{O}_{\P^3}(2))$ fails to satisfy property $\QR(3)$ since $W'$ is a proper subspace of $I(V_{3,2})_2$. To see this claim, consider the ring homomorphism
\begin{equation*}
\nu_2^{\#}: \Bbbk [z_0 , z_1 , \ldots ,
z_9 ] \to \Bbbk[x_0,\ldots,x_3]
\end{equation*}
corresponding to the Veronese embedding $\P^3\lhook\joinrel\xrightarrow{~\nu_2} V_{3,2}\subset\P^9$. Note that there is no rank $2$ element in $I(V_{3,2})$. If there exists a nonzero rank $2$ element $L_1^2+L_2^2\in I(V_{3,2})$, then $\nu_2^{\#}(L_1^2+L_2^2)=q_1^2+q_2^2=0$ and since $\Bbbk[x_0,\ldots,x_3]$ is a UFD, it means that $q_1=\pm i q_2$ for $i \in \Bbbk$ such that $i^2=-1$. This implies that $L_1=\pm i L_2$ and it contradict to the assumption that $L_1^2+L_2^2$ is nonzero. Let $\Theta = L_1^2-L_2L_3 \in I(V_{3,2})$ be a nonzero quadric of rank $3$ where $L_1 , L_2 ,
L_3 \in \Bbbk [z_0 , z_1 , \ldots , z_9 ]$ are linear forms. We may assume that every rank $3$ quadric is of this form via the transform $L_1^2+L_2^2+L_3^2\mapsto L_1^2-(L_2+i L_3)(-L_2+i L_3)$. Then $\nu_2^{\#}(L_1)=q_1, \nu_2^{\#}(L_2)=q_2, \nu_2^{\#}(L_3)=q_3$ for some quadratic forms $q_1 , q_2 , q_3$ in $\Bbbk[x_0,\ldots,x_3]$ such that
\begin{equation*}
 \nu_2^{\#}(\Theta) = q_1 ^2 - q_2 q_3 = 0.
\end{equation*}
We can check that $q_1$ is not proportional to $q_2$ (and also to $q_3$) since if it is, then the original quadric $L_1^2-L_2L_3$ has rank $2$ which is a contradiction. From the UFD property of  $\Bbbk[x_0,\ldots,x_3]$, it follows that $q_1=l_1l_2$, $q_2=l_1^2$ and $q_3=l_2^2$ for some linear forms $l_1 ,l_2 , l_3$. In consequence, $\Theta = -Q(l_1,l_2,1)$ is an element of $W$.

For $n \geq 4$, choose a hyperplane $H=\P^{n-1}$ of $\P^n$. Then one can check that
$V_{n-1,2} = \nu_2 (H)$ in $\P^{N(n-1,2)}$ is an ideal-theoretic linear section of
$V_{n,2} \subset \P^{N(n,2)}$ (cf. \cite[Proposition 2.1]{CP}). In particular, if $(\P^{n-1} , \mathcal{O}_{\P^{n-1}} (2))$ fails to satisfy property $\QR(3)$ then so does $(\P^{n} , \mathcal{O}_{\P^{n}} (2))$. This
completes the proof that $(\P^n, \mathcal{O}_{\P^n}(2))$ fails to
satisfy property $\QR(3)$ for all $n \geq 3$. \qed 

\begin{remark}
From Theorem \ref{thm:producingR3Qs}, one may ask if all the members of $\Gamma$ are $\Bbbk$-linearly independent or not.

\noindent $(1)$ For $(\P^2 , \mathcal{O}_{\P^2} (2))$, it is shown in the proof of Theorem \ref{thm:secondVeronese} that $\Gamma$ is a basis for $I(V_{2,2})$.

\noindent $(2)$ For $(\P^3 , \mathcal{O}_{\P^3} (2))$, suppose that ${\rm char}~\Bbbk \neq 2,3$. In the proof of Theorem \ref{thm:secondVeronese}, it is shown that $\Gamma$ spans the $20$-dimensional vector space $I(V_{3,2})_2$. Since $|\Gamma| \leq 21$, there exists at most one non-trivial $\Bbbk$-linear relation among the elements of $\Gamma$.
Note that the identity
\begin{equation}\label{eq:unique relation}
H'_{0,1,2,3}+H'_{0,2,1,3}+H'_{0,3,1,2}=0
\end{equation}
provides one such relation. In consequence, (\ref{eq:unique relation}) is essentially the unique linear relation among the elements of $\Gamma$.
\end{remark}
\smallskip

\section{Higher Veronese embeddings}\label{sect4}
\noindent This section is devoted to the proof of Theorem
\ref{thm:Veronese1}. So, we assume that ${\rm char}~\Bbbk \neq 2,3$.
We require a few more notation and definition.

\begin{notation and remarks}\label{additional not and rmk:Veronese}(Continued from Notation and Remarks \ref{not and rmk:Veronese}) Let
\begin{equation*}
V_{n,d} := \nu_d (\P^n ) \subset \P^{N } \quad \mbox{where}
\quad N= N(n,d) = {{n+d} \choose {n}} -1
\end{equation*}
be the $d$th Veronese embedding of $\P^n$ for some $n \geq 1$ and $d \geq 2$.

\renewcommand{\descriptionlabel}[1]%
             {\hspace{\labelsep}\textrm{#1}}
\begin{description}
\setlength{\labelwidth}{13mm} \setlength{\labelsep}{1.5mm}
\setlength{\itemindent}{0mm}

\item[{\rm (1)}] For $0 \leq i \leq n$, let $e_i$ denote the $i$th coordinate vector. Thus $A(n,1)=\{e_i ~|~ 0 \leq i \leq n \}$.
\item[{\rm (2)}] For $I \in  A(n,d)$, we define $\Supp(I)$ as the set $\{ k ~|~ 0 \leq k \leq n, ~I_k >0 \}$.
\item[{\rm (3)}] For simplicity, we denote the monomial $z_Iz_J$ by $[I,J]$. For example, an element $z_Iz_J-z_Kz_L$ of $\mathcal{Q} (n,d)$ is
denoted by $[I,J]-[K,L]$.
\item[{\rm (4)}] Throughout this section, we use the map $Q$, the ideal
$\mathfrak{Q}$ and the finite set $\Gamma := \Gamma (L_1 , L_2)$ in
Notation and Remarks \ref{nar:producingQofR3} for the pair $(L_1
,L_2 ) =(\mathcal{O}_{\P^n} (1),\mathcal{O}_{\P^n} (d-2))$. By
Theorem \ref{thm:producingR3Qs}, we know that $\mathfrak{Q}$ is
generated by $\Gamma$.
\item[{\rm (5)}] Let $q$ and $q'$ be two homogeneous quadratic equations. We say that
they are equivalently related and write $q\sim q'$ if $q-q'$ can be
represented by the sum of elements in $\Gamma$, or equivalently,
$q-q'$ is contained in $\mathfrak{Q}$. So a quadratic equation $q$
is equivalent to $0$ if and only if $q \in \mathfrak{Q}$.
\item[{\rm (6)}] From Notation and Remarks \ref{not and rmk:Veronese}.(1), one can
see that the following three statements are equivalent.
\begin{enumerate}
\item[$(i)$] $[I,J] \sim [K,L]$ for any $I,J,K,L \in A(n,d)$ satisfying $I+J = K+L$;
\item[$(ii)$] $I(V_{n,d}) = \mathfrak{Q}$;
\item[$(iii)$] $I(V_{n,d})$ is generated by $\Gamma$.
\end{enumerate}
In particular, $(\P^n , \mathcal{O}_{\P^n} (d))$ satisfies property
$\QR(3)$ if one of $(i) \sim (iii)$ holds.
\end{description}
\end{notation and remarks}

We will prove the following

\begin{theorem}\label{thm:finite generating set of rank 3 quadrics}
$I(V_{n,d})$ is generated by $\Gamma = \Gamma (\mathcal{O}_{\P^n}
(1),\mathcal{O}_{\P^n} (d-2))$.
\end{theorem}

As is mentioned in Notation and Remarks \ref{additional not and
rmk:Veronese}.(6), Theorem \ref{thm:finite generating set of rank 3
quadrics} implies Theorem \ref{thm:Veronese1}.

In order to prove Theorem \ref{thm:finite generating set of rank 3
quadrics}, we will use the double induction on $n$ and $d$. Indeed,
Theorem \ref{thm:finite generating set of rank 3 quadrics} is
already proven for $n=1$ in Corollary \ref{cor:rational normal curve} and
for $d=2$ in Theorem \ref{thm:secondVeronese}, respectively. From now on, let $n \geq 2$ and $d\geq 3$. Also we assume that \\

\begin{enumerate}
\item[$(\dagger)$] the statement of Theorem \ref{thm:finite generating set of rank 3 quadrics}
is true for $(\P^{n'} , \mathcal{O}_{\P^{n'}} (d'))$ whenever $n' < n$ and $d' < d$. \\
\end{enumerate}

\noindent Assuming these, we will first prove some lemmas.

\begin{lemma}\label{lem:gamma}
Suppose that $d \geq 3$ and let $I$ and $J$ be two elements in $A(n,d-2)$. Then

\renewcommand{\descriptionlabel}[1]%
             {\hspace{\labelsep}\textrm{#1}}
\begin{description}
\setlength{\labelwidth}{13mm} \setlength{\labelsep}{1.5mm}
\setlength{\itemindent}{0mm}
\item[$(1)$] $[2e_i+I,2e_j+I]\sim [e_i+e_j+I,e_i+e_j+I]$
\item[$(2)$] $[2e_i+I,e_j+e_k+I]\sim [e_i+e_j+I,e_i+e_k+I]$
\item[$(3)$] $[e_i+e_j+I,e_k+e_l+I]\sim [e_i+e_k+I,e_j+e_l+I]$
\item[$(4)$] $[2e_i+I,2e_j+J]+[2e_j+I,2e_i+J]\sim 2[e_i+e_j+I,e_i+e_j+J]$
\item[$(5)$] $[2e_i+I,e_j+e_k+J]+[e_j+e_k+I,2e_i+J]$

$\quad \quad \quad \quad \quad \quad \quad  \quad \quad \quad \quad
\quad \sim [e_i+e_j+I,e_i+e_k+J]+[e_i+e_k+I,e_i+e_j+J]$
\item[$(6)$] $[e_i+e_j+I,e_k+e_l+J]+[e_k+e_l+I,e_i+e_j+J]$

$\quad \quad \quad \quad \quad \quad \quad  \quad \quad \quad \quad
\quad \sim [e_i+e_k+I,e_j+e_l+J]+[e_j+e_l+I,e_i+e_k+J]$
\end{description}
\end{lemma}

\begin{proof}
To use the $Q$-map about the decomposition $\mathcal{O}_{\P^n} (d) =
\mathcal{O}_{\P^n} (1)^{\otimes 2} \otimes \mathcal{O}_{\P^n}
(d-2)$, let $\{x_0 , x_1 , \ldots , x_n \}$ be a basis of $H^0 (\P^n
,\mathcal{O}_{\P^n} (1))$.
\smallskip

\noindent $(1)$ For an element in $\Gamma_{11}$,
\begin{align*}
Q(x_i,x_j,\mathbf{x}^I) & \quad = \quad  f(x_i^2\mathbf{x}^I)f(x_j^2\mathbf{x}^I)-f(x_ix_j\mathbf{x}^I)^2 \\
                        & \quad = \quad  [2e_i+I,2e_j+I]-[e_i+e_j+I,e_i+e_j+I]\in \mathfrak{Q}
\end{align*}
Hence, the relation $(1)$ holds.

\noindent $(2)$ For an element in $\Gamma_{12}$,
\begin{align*}
&Q(x_i,x_j+x_k,\mathbf{x}^I)  =  f(x_i^2\mathbf{x}^I)f((x_j+x_k)^2\mathbf{x}^I)-f(x_i(x_j+x_k)\mathbf{x}^I)^2  \\
                         &\quad =   [2e_i+I,2e_j+I]+2[2e_i+I,e_j+e_k+I]+[2e_i+I,2e_k+I] \\
                        &\quad\quad  -[e_i+e_j+I,e_i+e_j+I]-2[e_i+e_j+I,e_i+e_k+I] -[e_i+e_k+I,e_i+e_k+I] .
\end{align*}
Using the equivalence relation $(1)$ for each pair $(i,j)$ and $(i,k)$, this is equivalent to $2[2e_i+I,e_j+e_k+I]-2[e_i+e_j+I,e_i+e_k+I]$. Since ${\rm char}(\Bbbk) \neq 2$, it follows the above relation $(2)$.

\noindent $(3)$  For an element in $\Gamma_{22}$,
\begin{align*}
& Q(x_i+x_j,x_k+x_l,\mathbf{x}^I)  =     f((x_i+x_j)^2\mathbf{x}^I)f((x_k+x_l)^2\mathbf{x}^I)-f((x_i+x_j)(x_k+x_l)\mathbf{x}^I)^2 \\
&=[2e_i+I,2e_k+I]-[e_i+e_k+I,e_i+e_k+I]+[2e_i+I,2e_l+I]\\
&\quad -[e_i+e_l+I,e_i+e_l+I]+[2e_j+I,2e_k+I]-[e_j+e_k+I,e_j+e_k+I]\\
&\quad +[2e_j+I,2e_l+I]-[e_j+e_l+I,e_j+e_l+I]+2[2e_i+I,e_k+e_l+I]\\
&\quad -2[e_i+e_k+I,e_i+e_l+I]+2[2e_j+I,e_k+e_l+I]-2[e_j+e_k+I,e_j+e_l+I]\\
&\quad +2[e_i+e_j+I,2e_k+I]-2[e_i+e_k+I,e_j+e_k+I]+2[e_i+e_j+I,2e_l+I]\\
&\quad -2[e_i+e_l+I,e_j+e_l+I]+4[e_i+e_j+I,e_k+e_l+I]-2[e_i+e_k+I,e_j+e_l+I]\\
&\quad -2[e_i+e_l+I,e_j+e_k+I]~.
\end{align*}
We can erase the above $8$ equations using the relation $(1)$ and $(2)$.
Then the only remaining equivalent part is
\begin{equation}\label{eq31}
4[e_i+e_j+I,e_k+e_l+I]-2[e_i+e_k+I,e_j+e_l+I]-2[e_i+e_l+I,e_j+e_k+I].
\end{equation}
By exchanging $j$ and $k$, we also obtain the relation
\begin{equation}\label{eq32}
Q(x_i+x_k,x_j+x_l,\mathbf{x}^I)
\sim4[e_i+e_k+I,e_j+e_l+I]-2[e_i+e_j+I,e_k+e_l+I]-2[e_i+e_l+I,e_j+e_k+I].
\end{equation}
By subtracting (\ref{eq32}) from (\ref{eq31}) and dividing it by $6$, it can be shown that
\begin{equation*}
[e_i+e_j+I,e_k+e_l+I]\sim [e_i+e_l+I,e_j+e_k+I].
\end{equation*}
To this aim, we require the assumption that ${\mbox char}(\Bbbk) \neq 2,3$.

\noindent $(4)$ For an element in $\Gamma_{11}$,
\begin{align*}
&Q(x_i,x_j,\mathbf{x}^I+\mathbf{x}^J) = f(x_i^2(\mathbf{x}^I+\mathbf{x}^J))f(x_j^2(\mathbf{x}^I+\mathbf{x}^J))-f(x_ix_j(\mathbf{x}^I+\mathbf{x}^J))^2\\
&\quad=  [2e_i+I,2e_j+I]+[2e_i+J,2e_j+J]+[2e_i+I,2e_j+J]+[2e_i+J,2e_j+I]\\
& \quad  -[e_i+e_j+I,e_i+e_j+I]-[e_i+e_j+J,e_i+e_j+J]-2[e_i+e_j+I,e_i+e_j+J]~.
\end{align*}
Using the relation $(1)$ for $I$ and $J$, one can show that the above equation is equivalent to
\begin{equation*}
[2e_i+I,2e_j+J]+[2e_j+I,2e_i+J]- 2[e_i+e_j+I,e_i+e_j+J].
\end{equation*}
Hence the relation $(4)$ holds.

\noindent $(5)$ For an element in $\Gamma_{12}$,
\begin{align*}
&Q(x_i,x_j+x_k,\mathbf{x}^I+\mathbf{x}^J) \\
&= f(x_i^2(\mathbf{x}^I+\mathbf{x}^J))f((x_j+x_k)^2(\mathbf{x}^I+\mathbf{x}^J))-f(x_i(x_j+x_k)(\mathbf{x}^I+\mathbf{x}^J))^2\\
&= [2e_i+I,2e_j+I]-[e_i+e_j+I,e_i+e_j+I]\\
&+[2e_i+I,2e_k+I]-[e_i+e_k+I,e_i+e_k+I]\\
&+[2e_i+J,2e_j+J]-[e_i+e_j+J,e_i+e_j+J]\\
&+[2e_i+J,2e_k+J]-[e_i+e_k+J,e_i+e_k+J]\quad\textrm{(first 4 lines $\sim$ $0$ by (1))}\\
&+2[2e_i+I,e_j+e_k+I]-2[e_i+e_j+I,e_i+e_k+I]\\
&+2[2e_i+J,e_j+e_k+J]-2[e_i+e_j+J,e_i+e_k+J]\quad\textrm{(next two lines $\sim$ $0$ by (2))}\\
&+[2e_i+I,2e_j+J]+[2e_i+J,2e_j+I]-2[e_i+e_j+I,e_i+e_j+J]\\
&+[2e_i+I,2e_k+J]+[2e_i+J,2e_k+I]-2[e_i+e_k+I,e_i+e_k+J]\quad\textrm{(and two lines $\sim$ $0$ by (4))}\\
&+2[2e_i+I,e_j+e_k+J]+2[2e_i+J,e_j+e_k+I]-2[e_i+e_j+I,e_i+e_k+J]\\
&-2[e_i+e_k+I,e_i+e_j+J] \quad\textrm{(the only remaining part)}
\end{align*}
Thus, we have
\begin{equation*}
[2e_i+I,e_j+e_k+J]+[e_j+e_k+I,2e_i+J]-[e_i+e_j+I,e_i+e_k+J]-[e_i+e_k+I,e_i+e_j+J]\in \mathfrak{Q}~.
\end{equation*}
Hence the relation $(5)$ holds.

\noindent $(6)$ For an element in $\Gamma_{22}$,
\begin{footnotesize}
\begin{align*}
& Q(x_i+x_j,x_k+x_l,\mathbf{x}^I+\mathbf{x}^J)\\
& =  f((x_i+x_j)^2(\mathbf{x}^I+\mathbf{x}^J))f((x_k+x_l)^2(\mathbf{x}^I+\mathbf{x}^J))  -f((x_i+x_j)(x_k+x_l)(\mathbf{x}^I+\mathbf{x}^J))^2\\
&= [2e_i+I,2e_k+I]-[e_i+e_k+I,e_i+e_k+I]+[2e_i+I,2e_l+I]-[e_i+e_l+I,e_i+e_l+I]\\
&+[2e_j+I,2e_k+I]-[e_j+e_k+I,e_j+e_k+I]+[2e_j+I,2e_l+I]-[e_j+e_l+I,e_j+e_l+I]\\
&+[2e_i+J,2e_k+J]-[e_i+e_k+J,e_i+e_k+J]+[2e_i+J,2e_l+J]-[e_i+e_l+J,e_i+e_l+J]\\
&+[2e_j+J,2e_k+J]-[e_j+e_k+J,e_j+e_k+J]+[2e_j+J,2e_l+J]-[e_j+e_l+J,e_j+e_l+J]\\
&\quad\quad\textrm{(first 4 lines equivalent to $0$ by (1))}\\
&+2[2e_i+I,e_k+e_l+I]-2[e_i+e_k+I,e_i+e_l+I]+2[2e_i+J,e_k+e_l+J]-2[e_i+e_k+J,e_i+e_l+J]\\
&+2[2e_j+I,e_k+e_l+I]-2[e_j+e_k+I,e_j+e_l+I]+2[2e_j+J,e_k+e_l+J]-2[e_j+e_k+J,e_j+e_l+J]\\
&+2[e_i+e_j+I,2e_k+I]-2[e_i+e_k+I,e_j+e_k+I]+2[e_i+e_j+J,2e_k+J]-2[e_i+e_k+J,e_j+e_k+J]\\
&+2[e_i+e_j+I,2e_l+I]-2[e_i+e_l+I,e_j+e_l+I]+2[e_i+e_j+J,2e_l+J]-2[e_i+e_l+J,e_j+e_l+J]\\
&\quad\quad\textrm{(next 4 lines equivalent to $0$ by (2))}\\
&+[2e_i+I,2e_k+J]+[2e_i+J,2e_k+I]-2[e_i+e_k+I,e_i+e_k+J]\\
&+[2e_i+I,2e_l+J]+[2e_i+J,2e_l+I]-2[e_i+e_l+I,e_i+e_l+J]\\
&+[2e_j+I,2e_k+J]+[2e_j+J,2e_k+I]-2[e_j+e_k+I,e_j+e_k+J]\\
&+[2e_j+I,2e_l+J]+[2e_j+J,2e_l+I]-2[e_j+e_l+I,e_j+e_l+J]\\
&\quad\quad\textrm{(and 4 lines equivalent to $0$ by (4))}
\end{align*}
\begin{align*}
&+2[2e_i+I,e_k+e_l+J]+2[2e_i+J,e_k+e_l+I]-2[e_i+e_k+I,e_i+e_l+J]-2[e_i+e_l+I,e_i+e_k+J]\\
&+2[2e_j+I,e_k+e_l+J]+2[2e_j+J,e_k+e_l+I]-2[e_j+e_k+I,e_j+e_l+J]-2[e_j+e_l+I,e_j+e_k+J]\\
&+2[e_i+e_j+I,2e_k+J]+2[e_i+e_j+J,2e_k+I]-2[e_i+e_k+I,e_j+e_k+J]-2[e_j+e_k+I,e_i+e_k+J]\\
&+2[e_i+e_j+I,2e_l+J]+2[e_i+e_j+J,2e_l+I]-2[e_i+e_l+I,e_j+e_l+J]-2[e_j+e_l+I,e_i+e_l+J]
\\
&\quad\quad\textrm{(and 4 lines equivalent to $0$ by (5))}\\
&+4[e_i+e_j+I,e_k+e_l+I]-2[e_i+e_k+I,e_j+e_l+I]-2[e_i+e_l+I,e_j+e_k+I]\\
&+4[e_i+e_j+J,e_k+e_l+J]-2[e_i+e_k+J,e_j+e_l+J]-2[e_i+e_l+J,e_j+e_k+J]\\
&\quad\quad\textrm{(and 2 lines equivalent to $0$ by (3))}\\
&+4[e_i+e_j+I,e_k+e_l+J]+4[e_k+e_l+I,e_i+e_j+J]-2[e_i+e_k+I,e_j+e_l+J]\\
&-2[e_j+e_l+I,e_i+e_k+J]-2[e_i+e_l+I,e_j+e_k+J]-2[e_j+e_k+I,e_i+e_l+J]\quad.\\
&\quad\quad\textrm{(the only remaining part)}
\end{align*}
\end{footnotesize}
Hence we see that $Q(x_i+x_j,x_k+x_l,\mathbf{x}^I+\mathbf{x}^J)$ is equivalent to
\begin{equation}\label{eq61}
\begin{CD}
4[e_i+e_j+I,e_k+e_l+J]+4[e_k+e_l+I,e_i+e_j+J] -2[e_i+e_k+I,e_j+e_l+J] \\
-2[e_j+e_l+I,e_i+e_k+J]-2[e_i+e_l+I,e_j+e_k+J]-2[e_j+e_k+I,e_i+e_l+J].
\end{CD}
\end{equation}
By exchanging $j$ and $k$, we also obtain the relation
\begin{equation}\label{eq62}
\begin{CD}
Q(x_i+x_k,x_j+x_l,\mathbf{x}^I+\mathbf{x}^J)\sim\\
4[e_i+e_k+I,e_j+e_l+J]+4[e_j+e_l+I,e_i+e_k+J]-2[e_i+e_j+I,e_k+e_l+J] \\
 -2[e_k+e_l+I,e_i+e_j+J]-2[e_i+e_l+I,e_j+e_k+J]-2[e_j+e_k+I,e_i+e_l+J].
\end{CD}
\end{equation}
By subtracting (\ref{eq62}) from (\ref{eq61}) and dividing it by $6$, it can be shown that
$$[e_i+e_j+I,e_k+e_l+J]+ [e_k+e_l+I,e_i+e_j+J]\sim [e_i+e_k+I,e_j+e_l+J]+[e_j+e_l+I,e_i+e_k+J].$$
Here, we require again the assumption that ${\rm char}(\Bbbk) \neq 2,3$.
\end{proof}

\begin{lemma}\label{lem:induction}
Let $I,J,K, L \in A(n,d)$ be such that $I+J=K+L$. Under the assumption $(\dagger)$ (just after Theorem \ref{thm:finite generating set of rank 3 quadrics}), if
\begin{equation*}
\Supp(I)\cap \Supp(J)\cap \Supp(K) \cap \Supp(L)~\mbox{or}~\Supp(I)^c\cap \Supp(J)^c\cap \Supp(K)^c \cap \Supp(L)^c
\end{equation*}
is nonempty, then $[I,J]\sim [K,L]$.
\end{lemma}

\begin{proof}
The case where $\Supp(I)^c\cap \Supp(J)^c\cap \Supp(K)^c \cap
\Supp(L)^c$ is nonempty is already dealt with in the proof of
Proposition \ref{prop:Veronese lower d}, by the induction on $n$.

Next, suppose that $\Supp(I)\cap \Supp(J)\cap \Supp(K) \cap \Supp(L)$ has an element, say $i$. Note that the four elements
\begin{equation*}
I'=I-e_i,\quad J'=J-e_i,\quad K'=K-e_i \quad \mbox{and} \quad L'=L-e_i
\end{equation*}
of $A(n,d-1)$ satisfies the condition $I'+J'=K'+L'$. By the
induction hypothesis on $d$, it holds that $(I',J')\sim (K',L')$. Now, by
using the map $\delta_i:I(V_{n,d-1})_2\rightarrow I(V_{n,d})_2$ induced
from multiplication by $x_i$, we can conclude that $[I,J]\sim
[K,L]$.
\end{proof}

Using the above two lemmas, we can prove the following lemmas.

\begin{lemma}\label{lem:311case}
Suppose that $I_0\geq 3$ and $J_1\geq 1, J_2\geq 1$. Then
\begin{equation*}
[I,J]\sim [-2e_0+e_1+e_2+I,2e_0-e_1-e_2+J].
\end{equation*}
\end{lemma}

\begin{proof}
By the Lemma \ref{lem:gamma} (5), we have

\noindent $[2e_0+(I-2e_0),e_1+e_2+(J-e_1-e_2)]+[e_1+e_2+(I-2e_0),2e_0+(J-e_1-e_2)]   \sim$
\begin{equation*}
[e_0+e_1+(I-2e_0),e_0+e_2+(J-e_1-e_2)]+[e_0+e_2+(I-2e_0),e_0+e_1+(J-e_1-e_2)].
\end{equation*}
Since $I_0-2e_0>0$, it holds by Lemma \ref{lem:induction} that
\begin{equation*}
[e_1+e_2+(I-2e_0),2e_0+(J-e_1-e_2)]\sim [e_0+e_1+(I-2e_0),e_0+e_2+(J-e_1-e_2)]
\end{equation*}
and
\begin{equation*}
[e_1+e_2+(I-2e_0),2e_0+(J-e_1-e_2)]\sim [e_0+e_2+(I-2e_0),e_0+e_1+(J-e_1-e_2)].
\end{equation*}
In consequence,
\begin{equation*}
[I,J]=[2e_0+(I-2e_0),e_1+e_2+(J-e_1-e_2)]\sim [e_1+e_2+(I-2e_0),2e_0+(J-e_1-e_2)].
\end{equation*}
Here we use the property that if $a+b\sim c+d, b\sim c$, and $b\sim d$, then $a\sim b$.
\end{proof}

\begin{lemma}\label{lem:1221case}
If $I_0,I_1,J_2,J_3\geq 1$ and $I_1$ or $J_2$ is $\geq 2$, then
\begin{equation*}
[I,J] \sim [-e_0+e_3+I,e_0-e_3+J].
\end{equation*}
\end{lemma}

\begin{proof}
Let $I'=e_2+e_3+(I-e_0-e_1)$ and $J'=e_0+e_1+(J-e_2-e_3)$. Then it follows by Lemma \ref{lem:gamma}.(6) that
\begin{align*}
[I,J]+ [I',J'] &= [e_0+e_1+(I-e_0-e_1),e_2+e_3+(J-e_2-e_3)]\\
&  \quad +[e_2+e_3+(I-e_0-e_1),e_0+e_1+(J-e_2-e_3)] \\
& \sim  [e_0+e_2+(I-e_0-e_1),e_1+e_3+(J-e_2-e_3)]\\
&  \quad +[e_1+e_3+(I-e_0-e_1),e_0+e_2+(J-e_2-e_3)] \\
&= [-e_1+e_2+I,e_1-e_2+J]+[-e_0+e_3+I,e_0-e_3+J].
\end{align*}
Since $I_1\geq2$ or $J_2\geq 2$, it holds that $I'_1=I_1-1\geq 1$ and $J'_1=J_1+1\geq 1$, or $I'_2=I_2+1\geq 1$ and $J'_2=J_2-1\geq 1$.
Therefore we have
$$[I',J']\sim [-e_1+e_2+I,e_1-e_2+J]$$
by Lemma \ref{lem:induction}. In conclusion, it is shown that
\begin{equation*}
[I,J]\sim [-e_0+e_3+I,e_0-e_3+J]
\end{equation*}
by the above equivalence relation of $[I,J]+ [I',J']$.
\end{proof}

Recall that the automorphism group ${\rm Aut} (V_{n,d} , \P^N )$ of the Veronese variety $V_{n,d}$ in $\P^N$ is
defined as
\begin{equation*}
{\rm Aut} (V_{n,d} , \P^N ) := \{ \sigma \in \PGL_N (\kk) ~|~ \sigma (V_{n,d}
) = V_{n,d} \}.
\end{equation*}
There is a natural group action of ${\rm Aut} (V_{n,d} , \P^N )$ on
the homogeneous ideal $I(V_{n,d} )$. In particular, it acts on the
$\kk$-vector space $I(V_{n,d})_2$ and the rank is preserved under
this action. Since there is a natural isomorphism between
$\PGL_n(\kk)$ and ${\rm Aut} (V_{n,d} , \P^N )$, the group
$\PGL_n(\kk)$ acts on the homogeneous coordinate rings of $\P^n$ and
$\P^N$. Also, it acts on $I(V_{n,d})_2$ by which the rank is
preserved. Moreover, this action commutes with the $Q$-map in the
sense that for any $\sigma \in \PGL_n(\kk)$,
\begin{equation*}
\sigma (Q(s,t,h)) = Q (\sigma (s) , \sigma (t) , \sigma (h)).
\end{equation*}
By using this observation, we can prove the following

\begin{lemma}\label{lem:1100case}
If $I_0=I_1 =K_0= K_1 = 1$ and $J_0= J_1=L_0= L_1=0$, then
$[I,J]\sim [K,L]$.
\end{lemma}

\begin{proof}
Let $I'=(2,0,I_2 ,\ldots, I_n )$, $I''=(0,2,I_2 ,\ldots, I_n )$,
$K'=(2,0,K_2 ,\ldots, K_n )$ and $K''=(0,2,K_2 ,\ldots, K_n )$. By
Lemma \ref{lem:induction}, we have
\begin{equation*}
[I',J]\sim[K',L] \quad \mbox{and} \quad [I'',J]\sim[K'',L].
\end{equation*}
Now, consider the automorphism $\sigma$ of $\P^n$ induced from the homogeneous coordinate change
\begin{equation*}
(x_0 , x_1 , \ldots , x_n ) \mapsto (x_0 + x_1 , x_1 , \ldots , x_n
).
\end{equation*}
From this change of coordinates $x_0^2$ is sent to $(x_0+x_1)^2=x_0^2+2x_0x_1+x_1^2$, which shows $\sigma(z_{I'})=z_{I'}+2z_I+z_{I''}$.
Then it holds that
\begin{equation*}
0 \sim \sigma ([I',J]-[K',L]) = [I',J]+2[I,J]+[I'',J]
 -[K',L]-2[K,L]-[K'',L].
\end{equation*}
In consequence, it holds that $[I,J]\sim[K,L]$.
\end{proof}

Now, we are ready to give a \\

\noindent {\bf Proof of Theorem \ref{thm:finite generating set of rank 3 quadrics}.} As mentioned
in Notation and Remarks \ref{additional not and rmk:Veronese}.(5),
we need to show that $[I,J]\sim [K,L]$ for any choice of $I,J,K,L
\in A(n,d)$ with $I+J=K+L$. There are the following three cases.
\smallskip

\renewcommand{\descriptionlabel}[1]%
             {\hspace{\labelsep}\textrm{#1}}
\begin{description}
\setlength{\labelwidth}{13mm} \setlength{\labelsep}{1.5mm}
\setlength{\itemindent}{0mm}
\item[Case 1.] $\Supp(I)\cap \Supp(J) \neq \emptyset$ and $\Supp(K) \cap \Supp(L)\neq\emptyset$
\item[Case 2.]  $\Supp(I)\cap \Supp(J) = \emptyset$ and $\Supp(K) \cap \Supp(L) \neq \emptyset$
\item[Case 3.]  $\Supp(I)\cap \Supp(J)=\emptyset$ and $\Supp(K) \cap \Supp(L) =\emptyset$
\end{description}
We will deal with these 3 cases in turn.\medskip

\noindent{\bf Case 1.} In Lemma \ref{lem:induction} we've already dealt with the case $\Supp(I)\cap\Supp(J)\cap\Supp(J)\cap\Supp(K)$ is nonempty. So, without loss of generality, we can assume that
\begin{equation*}
0\in\Supp(I)\cap\Supp(J) \quad \mbox{and} \quad 1\in\Supp(K)\cap\Supp(L).
\end{equation*}
If $K_0$ and $L_0$ are both nonzero, by Lemma \ref{lem:induction}, we are done. So, one of them is zero.
Similarly, one of $I_1$ and $J_1$ is zero. So, we may assume that
\begin{equation*}
J_1 = L_0 = 0 \quad \mbox{and hence} \quad K_0=I_0+J_0 \quad \mbox{and} \quad I_1=K_1 +L_1 .
\end{equation*}
Then it is true that
\begin{equation*}
J_2 + \cdots + J_n = d-J_0 = d- K_0 + I_0 = I_0 + K_1 + \cdots + K_n > K_1
\end{equation*}
since $I_0 > 0$. Thus there exist non-negative integers $J_2 ' , \ldots , J_n '$ such that
\begin{enumerate}
\item[$(i)$] $J_i \geq J_i '$ for every $2 \leq i \leq n$, and
\item[$(ii)$] $J_2 ' + \cdots + J_n ' = K_1$.
\end{enumerate}
Now, we define two elements $M$ and $N$ in $A(n,d)$ as
\begin{equation*}
M := (I_0 , L_1 , I_2 + J_2 ' , \ldots , I_n + J_n ' ) \quad \mbox{and}
\quad N := (J_0 , K_1 , J_2 - J_2 ' , \ldots , J_n - J_n ' ).
\end{equation*}
Then it holds that
\begin{equation*}
|M| = I_0 + I_2 + \cdots + I_n + L_1 + J_2 ' + \cdots + J_n ' = d-I_1 + L_1 + K_1 = d,
\end{equation*}
\begin{equation*}
|N| = J_0 + J_2 + \cdots + J_n +  K_1 - (J_2 ' + \cdots + J_n ' ) = J_0 + J_2 + \cdots + J_n = d
\end{equation*}
and
\begin{equation*}
M+N = (I_0 + J_0 , L_1 + K_1 , I_2 + J_2 , \ldots , I_n + J_n )   = I+J = K+L.
\end{equation*}
Therefore $[I,J]-[M,N]$ and $[K,L]-[M,N]$ are contained in $I(V_{n,d})$.
Also, $[I,J] \sim [M,N]$ (resp. $[K,L] \sim [M,N]$) since $I_0$, $J_0$, $M_0$ and $N_0$
(resp. $K_1$, $L_1$, $M_1$ and $N_1$) are nonzero (cf. Lemma \ref{lem:induction}). In consequence, it holds that $[I,J] \sim [K,L]$.\\

\noindent{\bf Case 2.} In this case, without loss of generality, we may assume that
$K_0\geq 1, L_0\geq 1$ and $I_0=K_0+L_0, J_0=0$. Then $I_0\geq 2$.
\smallskip

\begin{enumerate}
\item[] {\bf Case 2-1.} Suppose that $|Supp(I)|=1$. Since $I_0$ is already nonzero, $I_0=d \geq 3$ and $I_i=0$ for all $1\leq i\leq n$.
    \begin{enumerate}
    \item[] {\bf Case 2-1-1.} If $J_i =0$ for some $i\neq 0$, then $I_i=J_i=0$ and hence we are done by Lemma \ref{lem:induction}.
    \item[] {\bf Case 2-1-2.} Suppose that $J_i \neq 0$ for all $1\leq i\leq n$. Since $n \geq 2$, we have
$J_1\geq 1$ and $J_2\geq 1$. By Lemma \ref{lem:311case}, it holds that
\begin{equation*}
[I,J]\sim [I',J']=[-2e_0+e_1+e_2+I,2e_0-e_1-e_2+J].
\end{equation*}
Since $I'_0=I_0-2e_0 \geq 1$, $J'_0=J_0+2e_0 \geq 1$, $K_0 \geq 1$ and $L_0\geq1$,
it follows by Lemma \ref{lem:induction} that $[I',J']\sim [K,L].$ Therefore it is true that $[I,J]\sim [K,L]$.
    \end{enumerate}
\item[] {\bf Case 2-2.} Suppose that $| \Supp (I)| \geq 2$. Without loss of generality, we may assume that $I_1 \geq 1$ and $J_1 = 0$.
     \begin{enumerate}
    \item[] {\bf Case 2-2-1.} Suppose that $|\Supp(J)|=1$. Then we may assume that $J_2=d\geq 3$.
By Lemma \ref{lem:311case} with indices $I, J$ permuted,  it holds that
$$[I,J]\sim[-e_0-e_1+2e_2+I,e_0+e_1-2e_2+J].$$
Since $I_0-1\geq1$, $J_0+1\geq 1$, $K_0\geq 1$ and $L_0\geq 1$, it follows by Lemma \ref{lem:induction} that $[I,J]\sim[K,L]$.
    \item[] {\bf Case 2-2-2.} Suppose that $|\Supp(J)|\geq 2$. Then we may assume that $J_2\geq 1$ and $J_3\geq 1$.
    \begin{enumerate}
    \item[] {\bf Case 2-2-2-1.} Suppose that $I_1\geq 2$ or $J_2\geq 2$. By Lemma \ref{lem:1221case}, it holds that $[I,J]\sim[-e_0+e_3+I,e_0-e_3+J]$. Since $I_0-1\geq 1, J_0+1\geq1, K_0\geq 1$ and $L_0\geq1$, it follows by Lemma \ref{lem:induction} that $[I,J]\sim [K,L]$.
    \item[] {\bf Case 2-2-2-2.} Now, the only remaining case is where all nonzero entries of $I$ and $J$ other than $I_0$ are equal to $1$. After reordering, we can obtain the following form, where $\1_a$ and $\0_b$ mean respectively the list of $a$ $1$'s and $b$ $0$'s.
$$I=(I_0,\1_s,\1_t,\0_u,\0_v),~ J=(0,\0_s,\0_t,\1_u,\1_v)$$
$$K=(K_0,\1_s,\0_t,\1_u,\0_v), ~ L=(L_0,\0_s,\1_t,\0_u,\1_v)$$
\begin{enumerate}
\item[] {\bf Case 2-2-2-2-1.} If one of $s,t,u$ and $v$ is greater than or equal to $2$, then we can apply Lemma \ref{lem:1100case} to show that $[I,J]\sim [K,L]$.
\item[] {\bf Case 2-2-2-2-2.} If $s,t,u,v\leq 1$, then $d=|J|=u+v\leq 2$. So, it reduces to the case when $d\leq 2$ and it is contradict to the assumption.
\end{enumerate}
    \end{enumerate}
\end{enumerate}
\end{enumerate}

\noindent{\bf Case 3.} For the third case, by reordering the indices, we can represent $I,J,K,L$ by the tuples $S,T,U,V$ with nonzero entries with length $s,t,u,v$, respectively, as below:
\begin{equation}\label{case3_format}
I=(S,T,\0_u,\0_v),~ J=(\0_s,\0_t,U,V),~ K=(S,\0_t,U,\0_v),~ L=(\0_s,T,\0_u,V)
\end{equation}
where the maximum entry of each $S,T,U$ and $V$ is moved to the first place in each $S,T,U,V$ (namely $S_0, T_0, U_0$ and $V_0$). Note that we can freely exchange roles of $S,T,U,V$ in this format (\ref{case3_format}) as permuting indices.

\begin{enumerate}
\item[] {\bf Case 3-1.} Suppose that one of $s,t,u,v$ is zero. Without loss of generality, we may assume that $s=0$. Then $d=|I|=|T|$ and $d=|L|=|T|+|V|=d+|V|$. Hence $|V|=0$. It means that $I=(T,\0_u)=L$ and $J=(\0_t,U)=K$. This is a trivial case.
\item[] {\bf Case 3-2.} Suppose that $s,t,u,v\geq 1$. First, note that (after further re-indexing, if needed) we can write the relation $[I,J]\sim[K,L]$ as
\begin{align}\label{repres_IJKL}
\displayindent0pt
\displaywidth\textwidth
&[(S_0,T_0,0,0,\0_{n-3})+I',(0,0,U_0,V_0,\0_{n-3})+J']\nonumber\\
&\quad\sim[(S_0,0,U_0,0,\0_{n-3})+K',(0,T_0,0,V_0,\0_{n-3})+L']
\end{align}
where $I',J',K'$ and $L'$ are the remaining parts of $I,J,K$ and $L$ whose first 4 entries are all zeros. 
\begin{enumerate}
\item[] {\bf Case 3-2-1.} Now, let us treat the case of at least two of $S_0,T_0,U_0$ and $V_0$ being greater than or equal to $2$. As exchanging roles of $S,T,U,V$, we may assume that $S_0\geq 2, T_0\geq2$ or $S_0\geq 2, U_0\geq2$ in (\ref{case3_format}). For $S_0\geq 2, T_0\geq2$, by Lemma \ref{lem:1221case}, first we have
\begin{align*}\displayindent0pt
\displaywidth\textwidth
&[(S_0,T_0,0,0,\0_{n-3})+I',(0,0,U_0,V_0,\0_{n-3})+J']\\&\quad\sim [(S_0-1,T_0,0,1,\0_{n-3})+I',(1,0,U_0,V_0-1,\0_{n-3})+J'].
\end{align*}
And, again by Lemma \ref{lem:1221case} with indices permuted as $(0,1,2,3)\mapsto(0,2,1,3)$ we obtain 
\begin{align*}\displayindent0pt
\displaywidth\textwidth
&[(S_0,0,U_0,0,\0_{n-3})+K',(0,T_0,0,V_0,\0_{n-3})+L']\\&\quad\sim [(S_0-1,0,U_0,1,\0_{n-3})+K',(1,T_0,0,V_0-1,\0_{n-3})+L'].
\end{align*}
Since $S_0-1\geq1$, we have
\begin{align*}
&[(S_0-1,T_0,0,1,\0_{n-3})+I',(1,0,U_0,V_0-1,\0_{n-3})+J']\\&\quad\sim[(S_0-1,0,U_0,1,\0_{n-3})+K',(1,T_0,0,V_0-1,\0_{n-3})+L']
\end{align*}
by Lemma \ref{lem:induction} since the first entry belongs to the common support. From the expression (\ref{repres_IJKL}), using equivalence relation, we can deduce that $[I,J]\sim[K,L]$. For the other case $S_0\geq 2, U_0\geq2$, similarly we can apply Lemma \ref{lem:1221case} with permuted indices and Lemma  \ref{lem:induction} to obtain the same result.
\item[] {\bf Case 3-2-2.} Suppose that at most one of $S_0,T_0,U_0$ and $V_0$ is greater than or equal to $2$. Without loss of generality, we may assume that $T_0, U_0,V_0\leq 1$. Then, since the maximum entry is less than or equal to $1$, all the entries of $T,U,V$ are $1$'s in (\ref{case3_format}).
\begin{enumerate}
\item[] {\bf Case 3-2-2-1.} If $t,u$ or $v$ is greater than or equal to $2$ (i.e. at least two consecutive indices corresponding to $0$'s and $1$'s), then by re-indexing such indices to the first places, we can get that $[I,J]\sim[K,L]$ by Lemma \ref{lem:1100case}.
    \item[] {\bf Case 3-2-2-2.} When $t$, $u$ and $v$ are all less than or equal to $1$, we have $d=|J|=u+v\leq 2$. This contradicts to our assumption $n\ge2,~d\ge3$.
\end{enumerate}
\end{enumerate}
\end{enumerate}
This completes the proof. \qed

\section{Property $\QR(3)$ of arbitrary projective schemes}\label{sect5}
\noindent Our purpose in this section is to prove Theorem
\ref{thm:main1} about the asymptotic behavior of the rank of
quadratic generators of the Veronese re-embedding and to provide its
various applications.\\

\noindent {\bf Theorem \ref{thm:main1}.}
Suppose that ${\rm char}~\Bbbk \neq 2,3$ and let $L$ be a very ample
line bundle on a projective scheme $X$ defining the linearly normal
embedding
\begin{equation*}
X \subset \P H^0 (X,L).
\end{equation*}
If $m$ is an integer such that $X$ is $j$-normal for all $j \geq m$ and $I(X)$ is generated by forms of
degree $\leq m$, then $(X,L^d )$ satisfies property $\QR (3)$ for all $d
\geq m$.\\

Recall that $X$ is $j$-normal if the natural map $\textit{Sym}^j H^0(X,L) \rightarrow H^0(X,L^j)$ is surjective.\\

\noindent {\bf Proof of Theorem \ref{thm:main1}.} Let $X_d \subset
\P^{r(d)} := \P H^0 (X,L^d )$ be the linearly normal embedding of
$X$ by the complete linear series $|L^d|$. Since $X \subset \P^n :=
\P H^0 (X,L)$ is $d$-normal and $\P^{N(n,d)} = \P\textit{Sym}^d H^0(X,L)$, we can regard $\P^{r(d)}$ as a subspace of $\P^{N(n,d)}$ by the surjective map $\textit{Sym}^d H^0(X,L) \rightarrow H^0(X,L^d)$. Then $X_d$, the embedding of $X$ by $|L^d|$, is precisely equal to the image $\nu_d (X)$ where $\nu_d : \P^n \rightarrow \P^{N(n,d)}$ denotes the $d$th
Veronese embedding of $\P^n$ (cf. \cite[Exercise II.5.13]{H}). We
first show that $X_d$ is ideal-theoretically the intersection of
$V_{n,d} =\nu_d (\P^n )$ and $\P^{r(d)}$ in $\P^{N(n,d)}$. Let
$\mathcal{J}$ be the ideal sheaf of $V_{n,d}$ in $\P^{N(n,d)}$. Also
let $\mathcal{I}_{X_d}$ and $\mathcal{J}_{X_d}$ be respectively the
ideal sheaves of $X_d$ in $V_{n,d}$ and in $\P^{N(n,d)}$. Thus there
is an exact sequence
\begin{equation}\label{eq:ses}
0 \rightarrow \mathcal{J} \rightarrow \mathcal{J}_{X_d} \rightarrow
\mathcal{I}_{X_d} \rightarrow 0
\end{equation}
of coherent sheaves on $\P^{N(n,d)}$. Since $V_{n,d}$ is
projectively normal, we obtain the short exact sequence
\begin{equation*}
0 \rightarrow I(V_{n,d}/\P^{N(n,d)} ) \rightarrow I(X_d /
\P^{N(n,d)} ) \rightarrow E:= \bigoplus_{j \geq 0} H^0 (\P^n ,
\mathcal{I}_{X} (dj)) \rightarrow 0
\end{equation*}
of graded modules on $\P^{N(n,d)}$. Note that $E$ is generated by
$E_1$ as a graded module on $\P^{N(n,d)}$ since $I(X)$ is generated
by forms of degree $\leq d$. Also $I(\P^{r(d)} /\P^{N(n,d)})$ is
contained in $I(X_d / \P^{N(n,d)})$ and
\begin{equation*}
I(\P^{r(d)} /\P^{N(n,d)})_1 = I(X_d  / \P^{N(n,d)})_1 \cong E_1 .
\end{equation*}
Obviously, $I(\P^{r(d)}/\P^{N(n,d)})$ is generated by its degree one
piece. Therefore it holds that
\begin{equation*}
I(X_d / \P^{N(n,d)}) = I(V_{n,d} /\P^{N(n,d)} ) +
I(\P^{r(d)}/\P^{N(n,d)}),
\end{equation*}
which shows exactly that $X_d$ is the ideal-theoretic intersection
of $V_{n,d} =\nu_d (\P^n )$ and $\P^{r(d)}$.

Now, $X_d \subset \P^{r(d)}$ satisfies property $\QR(3)$ since the
Veronese variety $V_{n,d} \subset \P^{N(n,d)}$ satisfies property
$\QR(3)$ and $X_d$ is its linear section. \qed \\

\begin{remark}[An ideal-theoretic version of Mumford's fundamental observation in \cite{M}]\label{rmk:ideal version of Mumford's observation}
Let $X \subset \P^n$, $m$ and $X_d \subset \P^{r(d)}$ be as in
Theorem \ref{thm:main1} and its proof. So, it is shown in the above
proof that $X_d$ is an \textit{ideal-theoretic linear section} of
the Veronese variety $V_{n,d}$ for all $d \geq m$. Obviously, this
implies that various structural properties of $V_{n,d}$ is inherited
by $X_d$. For example, it is shown in \cite[Corollary 3.5]{Pu} that
$V_{n,d}$ is determinantally presented by $(1,d-1)$-type symmetric
flattening. Therefore $(X, L^d )$ is also determinantally presented for
all $d \geq m$. More precisely, let $\Omega(L,L^{d-1})$ be the
matrix of linear forms on $\P^{r(d)}$ obtained from the natural map
\begin{equation*}
H^0 (X,L) \otimes H^0 (X,L^{d-1}) \rightarrow H^0 (X,L^d ).
\end{equation*}
Then the homogeneous ideal of $X_d \subset \P^{r(d)}$ is generated by $2$-minors of $\Omega(L,L^{d-1})$ (cf. Theorem 1.1 in \cite{SS}).
\end{remark}

Theorem \ref{thm:main1} and Remark \ref{rmk:ideal version of
Mumford's observation} imply several geometric consequences for
which we are aiming.

\begin{corollary}\label{cor:Application 1}
Suppose that ${\rm char}~\Bbbk \neq 2,3$. Let $L$ be a very ample
line bundle on a projective scheme $X$ which satisfies condition
$\mathrm{N}_1$. Then $(X,L^d )$ is determinantally presented and
satisfies property $\QR(3)$ for all $d \geq 2$.
\end{corollary}

\begin{proof}
By our assumption on $L$, the linearly normal embedding $X \subset
\P H^0 (X,L)$ is $j$-normal for all $j \geq 2$ and its homogeneous
ideal is generated by forms of degree $\leq 2$. Thus the assertions
come immediately from Remark \ref{rmk:ideal version of Mumford's
observation} and Theorem \ref{thm:main1}, respectively.
\end{proof}

In a similar manner, we can obtain a more general statement as follows:

\begin{corollary}\label{cor:Application 2}
Suppose that ${\rm char}~\Bbbk \neq 2,3$. Let $A$ be an ample line
bundle on a projective connected scheme $X$. Then there exists a
number $d_0=d_0(X,A)$ such that if $d$ has a proper divisor $\geq d_0$, then
$A^d$ satisfies property $\QR(3)$. In particular, $A^d$ satisfies
property $\QR(3)$ if $d$ is even and $d \geq 2d_0$.
\end{corollary}

\begin{proof}
Note that a sufficiently ample line bundle on $X$ satisfies
condition $\mathrm{N}_1$ (cf. \cite[Proof of Theorem 1.1]{SS}). This
implies that there exists a number $d_0$ such that $A^d$ satisfies
condition $\mathrm{N}_1$ for every $d \geq d_0$. Now, suppose that
$d$ has a proper divisor $\ell \geq d_0$ and write $d = \ell \times
k$ where $k \geq 2$ is an integer. Then $A^d = (A^{\ell})^k$
satisfies property $\QR(3)$ since $A^{\ell}$ satisfies condition
$\mathrm{N}_1$ (cf. Corollary \ref{cor:Application 1}).
\end{proof}

We finish this section by providing a couple of applications of
Theorem \ref{thm:main1} to some classical varieties in the literature.

\begin{example}\label{ex:complete intersection}
Let $X \subset \P^n$ be a smooth complete intersection of
hypersurfaces $F_1 , \ldots , F_{c}$ of degrees $d_1 \geq \cdots
\geq d_c \geq 2$. Also let $L:=\mathcal{O}_X (1)$. Thus $X$ is
projectively normal and its homogeneous ideal is generated by forms
of degree $\leq d_1$.

\noindent $(1)$ By Theorem \ref{thm:main1}, it holds that $L^d$
satisfies property $\QR(3)$ for all $d \geq d_1$.

\noindent $(2)$ Recall that $K_X = \mathcal{O}_X (d_1 + \ldots + d_c
-n-1)$. Thus, if $c \geq 2$ and $d_2 + \cdots + d_c \geq n+1$ then
the canonical embedding of $X$ satisfies property $\QR(3)$ by $(1)$.

\noindent $(3)$ If $n \geq 5$ and $X$ is a curve, then it holds
always that $d_2 + \cdots + d_{n-1} \geq n+1$. Thus, by $(2)$, we
can see that there are $\infty$-many canonical curves which
satisfy property $\QR(3)$ (see also Example \ref{ex:canonical curve}).
\end{example}

\begin{example}\label{ex:Grassmannian}
Let $X={\rm Gr}(k,V)$ be the Grassmannian manifold of
$k$-dimensional subspaces of the $n$-dimensional $\Bbbk$-vector
space $V$ where $n \geq 3$ and $1 \leq k \leq n-2$. Also let $L$ be
the generator of ${\rm Pic}~X$ which defines the Pl\"ucker embedding
of $X$.

\noindent $(1)$ Let ${\rm Gr}(k,n) \subset \P (\bigwedge^k V)$ be
the Pl\"ucker embedding. Then it contains ${\rm Gr}(2,4)$ as a
linear section (cf. \cite[Chapter 6]{Harris}). Thus $L$ fails to
satisfy property $\QR(5)$ since ${\rm Gr}(2,4) \subset \P^5$ is a
hyperquadric of rank $6$. On the other hand, it is shown in
\cite{KPRS} that ${\rm Gr}(k,n)$ is \textit{set-theoretically} cut out by
quadratic equations of rank $6$.

\noindent $(2)$ It is well-known that $L$ satisfies condition
$\mathrm{N}_2$ (cf. Proposition 3.8 and Remark 3.9.(1) in
\cite{EGHP}). Therefore it follows by Theorem \ref{thm:main1} that
$({\rm Gr}(k,n),L^d )$ satisfies property $\QR(3)$ for all $d \geq
2$.
\end{example}

For the remaining part of this section, we assume that the
characteristic of $\Bbbk$ is zero.

\begin{example}\label{ex:Abelian}
Let $X$ be an abelian variety and let $A$ be an ample line bundle on
$X$. The main theorem in \cite{Pare} implies that $A^d$ satisfies
condition $\mathrm{N}_{d-3}$ for all $d \geq 4$. Then it follows by
Corollary \ref{cor:Application 2} and its proof that $(X,A^d )$
satisfies property $\QR(3)$ whenever $d \geq 8$ and neither $d \neq
9$ nor $d$ is a prime (that is, $d$ has a proper divisor $\geq 4$).
\end{example}

\begin{example}\label{ex:Enriques}
Let $S \subset \P^{g-1}$ be a linearly normal Enriques surface and
let $L:=\mathcal{O}_S (1)$. It is well-known that $g \geq 6$. Also
$(S,L)$ is called a \textit{Reye polarization} if $g=6$ and $S$
fails to be $2$-normal, or equivalently it lies on a quadric. In
\cite{GLM}, the authors prove the following results:
\begin{enumerate}
\item[$a.$] If $(S,L)$ is a Reye polarization, then $S \subset \P^5$ is $4$-regular and its homogeneous ideal is
generated by forms of degree $\leq 3$.
\item[$b.$] If $(S,L)$ is not a Reye polarization, then $S \subset  \P^{g-1}$ is $3$-regular.
\end{enumerate}
Thus Theorem \ref{thm:main1} shows that $(S,L^d )$ satisfies
property $\QR(3)$ for all $d \geq 3$.
\end{example}

\begin{example}\label{ex:K3}
Let $S \subset \P^{g}$ be a linearly normal K3 surface and let
$L:=\mathcal{O}_S (1)$. Recall that a general hyperplane section
$\mathcal{C} \subset \P^{g-1}$ of $S$ is a canonical curve of genus
$g$. Thus $S$ is projectively normal and its homogeneous ideal is
generated by quadratic and cubic equations. Now, Theorem
\ref{thm:main1} shows that $(S,L^d )$ satisfies property $\QR(3)$
for all $d \geq 3$. In particular, any general hyperplane section of
$S \subset \P H^0 (S,L^d)$ is a canonical curve satisfying property
$\QR(3)$.
\end{example}

\section{Further discussions and Open problems}\label{sect_eg_prbm}

\noindent In the final section, we will discuss some open questions related to our main results in the present paper.\\

\noindent{\textsf{Positivity and Rank $3$ quadratic generation~}}
Let $L$ be a very ample line bundle on a projective scheme $X$.
There have been many interesting results showing that the positivity
of $L$ is reflected in the defining equations of $X \subset \P H^0
(X,L)$ and the syzygies among them. To state more precisely, due to
\cite[Definition 3.1]{Gre2}, we say that a property $\mathcal{P}$
holds for every \textit{sufficiently ample} line bundle on $X$ if
there exists a line bundle $A$ on $X$ such that $\mathcal{P}$ holds
for all line bundles $L \in {\rm Pic}(X)$ for which $L \otimes
A^{-1}$ is ample. For example, let $p$ be a positive integer.
Theorem 1 in \cite{EL} shows that when $X$ is a smooth complex
variety, every sufficiently ample line bundle on $X$ satisfies
condition $\mathrm{N}_p$. Recently, Sidman and Smith in \cite{SS}
prove that every sufficiently ample line bundle on a connected
scheme is determinantally presented. In this direction, our main
results in this paper allude to a significant correlation
between some positive nature of the very ample line bundle $L$ on
$X$ and the property $\QR(3)$ of $(X,L)$, just like in many works on condition $\mathrm{N}_p$ and determinantal presentation of projective embeddings. For instance, a sufficiently ample line bundle on a projective connected scheme $X$ satisfies condition $\mathrm{N}_1$
(cf. \cite[Proof of Theorem 1.1]{SS}). For all those sufficiently
ample line bundles $L$, it follows by Corollary \ref{cor:Application
1} that $(X,L^2 )$ satisfies property $\QR (3)$. This observation
leads us naturally to formulate

\begin{conjecture}\label{con:Veronese variety arbitrary}
Every sufficiently ample line bundle on a projective scheme
satisfies property $\QR(3)$.
\end{conjecture}

\noindent $(1)$ The above conjecture is shown to be true by Theorem
\ref{thm:Veronese1} and Theorem \ref{thm:main1} when $\mbox{Pic}(X)$
is generated by a very ample line bundle (e.g. $X=\P^n$ or a
Grassmannian manifold).\\

\noindent $(2)$ In \cite{PaCurve}, the above conjecture is proved
for any projective integral curve $\mathcal{C}$ of arithmetic genus
$g$ by showing that property $\QR(3)$ holds for
$(\mathcal{C},\mathcal{L})$ whenever $\deg(\mathcal{L}) \ge 4g+4$.
Note that if $\mathcal{L} = \mathcal{M}^2$ and $\deg(\mathcal{M})
\geq 2g+2$, then $(\mathcal{C},\mathcal{L})$ satisfies property
$\QR(3)$ by Corollary \ref{cor:Application 1} since
$(\mathcal{C},\mathcal{M})$ satisfies condition $\mathrm{N}_{1}$.\\

\noindent $(3)$ Among aforementioned classical varieties, we intend
to treat Segre-Veronese embeddings and rational normal scrolls in
a forthcoming paper \cite{MP2}. Let $X = \P^{n_1} \times \cdots
\times \P^{n_k}$ where $k \geq 2$ and $n_1 , \ldots , n_k \geq 1$.
Every very ample line bundle on $X$ satisfies property $\QR(4)$. In
\cite{MP2}, we prove that $\mathcal{O}(d_1 , \ldots , d_k )$ (with
$d_1 , \ldots , d_k \ge 1$) fails to satisfy property $\QR(3)$ if
$d_i =1$ for some $1 \leq i \leq k$. This means that all line
bundles on the boundary of the ample cone of $X$ fail to satisfy
property $\QR(3)$.\\

\noindent{\textsf{Property $\QR(3)$ of canonical curves~}} Let
$\mathcal{C} \subset \P^{g-1}$ be a smooth canonical curve of genus
$g$ satisfying condition $\mathrm{N}_1$. M. Green in \cite{Gre}
proved that property $\QR(4)$ holds. Regarding property $\QR(3)$, we
have a computational example of some canonical curve of $g=6$ which
fail to satisfy property $\QR(3)$. See Example \ref{ex:canonical
curve}. On the other hand, our main results produce infinitely many
examples of canonical curves satisfying property $\QR(3)$ (cf.
Example \ref{ex:complete intersection} and Example \ref{ex:K3}).
Thus it seems an interesting open question whether property $\QR(3)$
holds for canonical curves of sufficiently large genus.

\begin{example}[A canonical curve without property $\QR(3)$]\label{ex:canonical curve}
Let $S:=\mathbb{K}[x_1,x_2,x_3,x_4,x_5,x_6]$ be the homogeneous coordinate ring of $\P^5$. Consider the homogeneous ideal $I =\langle Q_1 ,Q_2 ,Q_3 ,Q_4 ,Q_5 ,Q_6 \rangle$ of $S$ where
$$\begin{cases}
Q_1 =-x _2^2 +x _3 x _1 +x _3^2 -2x _4 x _2 +2x _5 x _1 +2x _5 x _3 -2x _5 x _4 -3x _6 x _1 -x _6 x _2\\ \quad\quad\quad -3x _6 x _3 +4x _6 x _4 + 3x _6 x _5 -8x _6^2 ,\\
Q_2 =-x _3 x _2 +2x _3^2 +x _4 x _1 -3x _4 x _2 +4x _5 x _1 +4x _5 x _3 +2x _5 x _4 +x _6 x _1 +5x _6 x _2\\ \quad\quad\quad -8x _6 x _3 +7x _6 x _4 +2x 6 x _5 +5x _6^2 ,\\
Q_3 = -2x_3^2 +2x_4 x_2 +x_4 x_3 -3x_5 x_1 -4x_5 x_3 -3x_5 x_4 -4x_6 x_1 -6x_6 x_2 \\
\quad\quad\quad +6x_6 x_3 -4x_6 x_4 -x_6 x_5 -10x_6^2 ,\\
Q_4 = -x_3 ^2 +x_4 x_2 +x_4 ^2 -x_5 x_1 -3x_5 x_3 -3x_5 x_4 -6x_6 x_1 -5x_6 x_2 +3x_6 x_3 \\
\quad\quad\quad-x_6 x_4 +x_6 x_5 -6 x_6 ^2\\
Q_5 = -x_3 ^2 +x_4 x_2 -x_5 x_1 +x_5 x_2 +2x_5 x_4 +7x_6 x_1 +3x_6 x_2 +4x_6 x_3 \\
\quad\quad\quad -6x_6 x_4 -2x_6 x_5 +6x_6 ^2\\
Q_6 = x_5^2 -x_6 x_1 +2x_6 x_2 +x_6 x_3 -2x_6 x_4~~.
\end{cases}$$
Then $\mathcal{C} := Z(I) \subset \P^5$ is a curve of genus $6$ and degree $10$. Indeed, it is the canonical embedding of the modular curve $X_0 (58)$. Furthermore, up to scalar multiplication, $Q_6$ is the only quadratic equation of rank $3$ in $I$. Therefore $\mathcal{C}$ fails to satisfy property $\QR(3)$. All computations are provided by the computer algebra systems \texttt{Macaulay2} \cite{GS}.
\end{example}

\end{document}